\documentclass[12pt]{amsart}

\usepackage{amsmath}
\usepackage{amscd}
\usepackage{amssymb}
\usepackage{amsfonts}
\usepackage{amsthm}
\usepackage{bbm}
\usepackage{cancel}
\usepackage{color}
\usepackage{eucal}
\usepackage{enumerate,yfonts}
\usepackage{accents}

\usepackage{pdfsync}
 \usepackage[all,cmtip]{xy}
\usepackage{graphicx}
\usepackage{graphics}
\usepackage{hyperref}
\usepackage{latexsym}
\usepackage{mathrsfs}
 \usepackage{placeins}
\usepackage{pstricks}
\usepackage{bookmark}
\usepackage{mathtools}

\usepackage{stmaryrd}
\usepackage{url}

\newtheorem{thm}{Theorem}[section]
\newtheorem{theorem}[thm]{Theorem}
\newtheorem{corollary}[thm]{Corollary}
\newtheorem{lemma}[thm]{Lemma}

\newtheorem{proposition}[thm]{Proposition}

\newtheorem{conjecture}[thm]{Conjecture}
\newtheorem{thm-dfn}[thm]{Theorem-Definition}

\newtheorem{example}[thm]{Example}
\newtheorem{definition}[thm]{Definition}
\newtheorem{remark}[thm]{Remark}

\numberwithin{equation}{section}

\newcommand{\nc}{\newcommand}

\newcommand{\cW}{{\mathcal W}}

\newcommand{\cO}{{\mathcal O}}
\newcommand{\cA}{{\mathcal A}}
\newcommand{\cF}{{\mathcal F}}
\newcommand{\cN}{{\mathcal N}}
\newcommand{\cH}{{\mathcal H}}
\newcommand{\cT}{{\mathcal T}}

\newcommand{\cP}{{\mathcal P}}
\newcommand{\cL}{{\mathcal L}}
\newcommand{\cE}{{\mathcal E}}

\newcommand{\cM}{{\mathcal{M}}}
\newcommand{\cG}{{\mathcal{G}}}
\newcommand{\cK}{{\mathcal{K}}}

\newcommand{\bC}{{\mathbb C}}
\newcommand{\bZ}{{\mathbb Z}}

\newcommand{\SO}{\mathrm{SO}}
\newcommand{\SL}{\mathrm{SL}}
\newcommand{\GL}{\mathrm{GL}}

\newcommand{\La}{{\mathfrak{a}}}

\newcommand{\Lg}{{\mathfrak g}}

\newcommand{\Ln}{{\mathfrak{n}}}

\newcommand{\Ll}{{\mathfrak{l}}}
\newcommand{\Lp}{{\mathfrak{p}}}

\newcommand{\fg}{{\mathfrak g}}

\newcommand{\fl}{{\mathfrak{l}}}
\newcommand{\fp}{{\mathfrak{p}}}
\newcommand{\fm}{{\mathfrak{m}}}

\newcommand{\fF}{{\mathfrak{F}}}

\nc{\II}{\mathrm{II}}
\nc{\I}{\mathrm{I}}

\newcommand{\p}{\perp}

\newcommand{\la}{\langle}
\newcommand{\ra}{\rangle}

\nc{\ot}{\otimes}
\nc{\on}{\operatorname}

\nc{\oh}{{\operatorname{H}}}
\nc{\gr}{{\operatorname{gr}}}
\nc{\rk}{{\operatorname{rank}}}
\nc{\codim}{{\operatorname{codim}}}
\nc{\img}{{\operatorname{Im}}}
\nc{\IC}{{\operatorname{IC}}}

\nc{\lp}{{\left(}}
\nc{\rp}{{\right)}}

\newcommand{\beqn}{\begin{equation*}}
\newcommand{\eeqn}{\end{equation*}}

\newcommand{\beq}{\begin{equation}}
\newcommand{\eeq}{\end{equation}}

\newcommand{\bern}{\begin{eqnarray*}}
\newcommand{\eern}{\end{eqnarray*}}

\begin{document}

\title{Springer correspondence for the split symmetric pair in type $A$}

\author{Tsao-Hsien Chen}
\email{chenth@math.uchicago.edu}
        \address{Department of Mathematics, University of Chicago, Chicago, 60637, USA}
        
        \thanks{Tsao-Hsien Chen was supported in part by the AMS-Simons travel grant and the NSF grant DMS-1702337.}
        \author{Kari Vilonen}
        \email{kari.vilonen@unimelb.edu.au}
        \address{School of Mathematics and Statistics, University of Melbourne, VIC 3010, Australia, and Department of Mathematics and Statistics, University of Helsinki, Helsinki, 00014, Finland}
         
         \thanks{Kari Vilonen was supported in part by  the ARC grants DP150103525 and DP180101445, the Academy of Finland, the Humboldt Foundation, the Simons Foundation,  and the NSF grant DMS-1402928.}
         \author{Ting Xue}
         \email{ting.xue@unimelb.edu.au}

         \address{ School of Mathematics and Statistics, University of Melbourne, VIC 3010, Australia, and Department of Mathematics and Statistics, University of Helsinki, Helsinki, 00014, Finland}
         \thanks{Ting Xue was supported in part by the ARC grants DE160100975, DP150103525 and the Academy of Finland.}


\begin{abstract}
In this paper we establish Springer correspondence for the symmetric pair $(\SL(N),\SO(N))$ using Fourier transform, parabolic induction functor, and a nearby cycle sheaf construction.
As an application of our results we see that the cohomology of Hessenberg varieties can be expressed in terms of  irreducible representations of Hecke algebras of symmetric groups at $q = -1$. Conversely, we see that the irreducible representations of Hecke algebras of symmetric groups at $q = -1$ arise in geometry.
\end{abstract}

\maketitle

\section{Introduction}

In this paper we consider the Springer correspondence in the case of symmetric spaces. We will concentrate on the split case of type $A$, i.e., the case of $\SL(n,\mathbb{R})$. The case of $\SL(n,\mathbb{H})$ was considered by Henderson in~\cite{H}, Grinberg in~\cite{G2} and Lusztig in \cite{L5}, and the case of $\mathrm{U}(p,q)$ was considered by Lusztig in~\cite{Lu3} where he treats the general case of semi-simple inner automorphisms. In both of these cases Springer theory closely resembles the classical situation. This turns out not to be so in the split case we consider here. In \cite{CVX1,CVX2} we have computed Fourier transforms of IC sheaves supported on certain nilpotent orbits using resolutions of singularities of nilpotent orbit closures. In this paper we study the problem in general in the split case of type $A$ replacing the resolutions with a nearby cycle sheaf construction in~\cite{GVX} based on earlier ideas of Grinberg \cite{G1,G2}. 

Let us call an irreducible IC sheaf supported on a nilpotent orbit  a nilpotent orbital complex. We show that the Fourier transform gives a bijection between nilpotent orbital complexes and certain representations of (extended) braid groups. We identify these representations of (extended) braid groups and construct them explicitly in terms of  irreducible representations of 
Hecke algebras of symmetric groups at $q=-1$. This bijection can be viewed as Springer correspondence for the symmetric pair $(\SL(N),\SO(N))$. Let us note that the fact that representations of (affine) Hecke algebras at $q=-1$ arise in this situation was already observed by Grojnowski in his thesis~\cite{Gr}. 

The proof of our main result, Theorem~\ref{main theorem}, makes use of 
a nearby cycle sheaf construction in~\cite{GVX} and smallness property of  
maps associated to certain $\theta$-stable parabolic subgroups.
The nearby cycle sheaves 
produce nilpotent orbital complexes whose Fourier transforms have full support.
Those IC sheaves behave like ``cuspidal sheaves'' in the sense that 
they do not appear as direct summands of parabolic inductions. 
On the other hand, the 
smallness property mentioned above implies a simple description of the images of 
parabolic induction functors 
(Proposition \ref{small}, Proposition \ref{prop-induction2}). Those results together with 
a counting lemma (Lemma \ref{lemma-count})
imply Theorem \ref{main theorem}. 
As corollaries,
we obtain criteria for nilpotent orbital complexes to have full support Fourier transforms (Corollary \ref{coro-full supp1}, Corollary \ref{coro-full supp2}) and results on cohomology of 
Hessenberg varieties (Theorem \ref{coh of Hess}). 

Our method appears to be applicable to
general symmetric pairs and, more generally, polar representations studied in \cite{G2}
and we
hope to return to this in future work. 

Let us mention that in \cite{LY}, the authors 
show that one can obtain all nilpotent orbital complexes 
using 
spiral induction functors introduced in \cite{LY} (in fact, they consider more general 
setting of cyclically graded Lie algebras). 
Using their results and Theorem \ref{main theorem}, we show that 
all irreducible representations of Hecke algebras of symmetric groups at $q=-1$
appear in the intersection cohomology of
of Hessenberg varieties, with coefficient in certain local systems 
(see Theorem \ref{geometric rel of rep}). This gives geometric constructions 
of irreducible representations of Hecke algebras of symmetric groups at $q=-1$ and provides them with a Hodge structure.

The paper is organized as follows. In Section~\ref{sec-pre} we 
recall some facts about symmetric pairs and  
introduce a class of 
representations of equivariant fundamental groups. Moreover, we recall the definition of the parabolic induction functor. 
In Section~\ref{sec-parabolic} we study parabolic induction functors for certain $\theta$-stable 
parabolic subgroups. In Section~\ref{sec-FT}, we prove Theorem \ref{main theorem}:
the Fourier transform defines a bijection between the set of nilpotent 
orbital complexes and the class of representations of equivariant fundamental groups 
introduced in Section~\ref{sec-pre}.  In Section~\ref{sec-Hess}
and Section~\ref{Rep}, we discuss applications of our results to 
cohomology of Hessenberg varieties and representations of Hecke algebras of symmetric groups  
at $q=-1$.
Finally, in Section~\ref{sec-conj}, we propose 
a conjecture that gives a more precise description of the bijection in Theorem \ref{main theorem}.

{\bf Acknowledgements.}
We thank Misha Grinberg for helpful conversations. TC thanks Cheng-Chiang Tsai for useful discussions and thanks the Institute 
of Mathematics Academia Sinica in Taipei for support, hospitality, and a nice research environment. TX thanks Syu Kato for helpful discussions. Furthermore, KV and TX thank the Research Institute for Mathematical Sciences in Kyoto  for support, hospitality, and a nice research environment. We also would like to thank the referee for carefully reading our manuscript and for helpful comments.

\section{Preliminaries}\label{sec-pre}
For convenience we work over $\bC$. We adopt the usual convention of cohomological degrees for perverse sheaves by having them be symmetric around 0. We also use the convention that all functors are derived, so we write, for example, $\pi_*$ instead of $R\pi_*$. If $X$ is smooth we write $\bC_X[-]$ for the constant sheaf placed in degree $-\dim X$ so that $\bC_X[-]$ is perverse. If $U\subset X$ is a smooth open dense subset of a variety $X$ and $\mathcal{L}$ is a local system on $U$, we write $\IC(X,\mathcal{L})$ for the IC-extension of $\cL[-]$ to $X$; in particular, it is perverse. For simplicity of notation, when we have a pair $(\cO,\cE)$, where $\cO$ is an orbit and $\cE$ is a local system on $\cO$, we also write $\IC(\cO,\cE)$ instead of $\IC(\bar\cO,\cE)$. 

\subsection{Notations} For $e\geq 2$, a partition $\lambda$ of a positive integer $k$ is called $e$-regular if the multiplicity of any part of $\lambda$ is {\em less than} $e$. In particular, a partition is $2$-regular if and only if it has distinct parts. Let us denote by $\cP(k)$ the set of all partitions of $k$ and by $\cP_2(k)$ the set of all $2$-regular  partitions of $k$.

We denote by $\cH_{k,-1}$ the Hecke algebra of the symmetric group $S_k$ with parameter $-1$. More precisely, $\cH_{k,-1}$ is the $\bC$-algebra generated by $T_i,\ i=1,\ldots,k-1$, with the following relations
\beq\label{hecke relations}
\begin{gathered}
T_iT_j=T_jT_i\text{ if }|i-j|\geq 2,\,i,j\in[1,k-1],\ T_iT_{i+1}T_i=T_{i+1}T_iT_{i+1},\,i\in[1,k-2],\\
T_i^2=q+(q-1)T_i,\text{ where }q=-1,\, i\in[1,k-1].
\end{gathered}
\eeq
It is shown in \cite{DJ} that the set of isomorphism classes of irreducible representations of $\cH_{k,-1}$ is parametrized by $\cP_2(k)$.  For $\mu\in\cP_2(k)$, we write $D_\mu$ for the irreducible representation of $\cH_{k,-1}$ corresponding to $\mu$.

For a real number $a$, we write $[a]$ for its integer part.

\subsection{The split symmetric pair \texorpdfstring{$(\SL(N),\SO(N))$}{lg}}Let $G=\SL(N)$ and $\theta:G\to G$ an involution such that $K=G^\theta=\SO(N)$ and write $\Lg=\on{Lie}G$. We have $\Lg=\Lg_0\oplus\Lg_1$, where $\theta|_{\Lg_i}=(-1)^i$. The pair $(G,K)$ is a split symmetric pair, i.e., there exists a maximal torus $A$ of $G$ that is $\theta$-split, where $\theta$-split means that for all $x\in A$, $\theta(x)=x^{-1}$.  We also think of the pair $(G,K)$ concretely as $(\SL(V),\SO(V))$, where $V$ is a vector space of dimension $N$ equipped with a non-degenerate quadratic form $Q$ such that $\SO(V)=\SO(V,Q)$. We write the non-degenerate bilinear form associated to $Q$ as $\langle\ ,\,\rangle$. 

Let $\Lg^{rs}$ denote the set of regular semisimple elements in $\Lg$ and let $\Lg_1^{rs}=\Lg_1\cap\Lg^{rs}$. Similarly, let $\Lg^{reg}$ denote the set of regular elements in $\Lg$ and let $\Lg_1^{reg}=\Lg_1\cap\Lg^{reg}$.

Let $\cN$ be the nilpotent cone of $\Lg$ and let $\cN_1=\cN\cap\Lg$. When $N$ is odd, the set of $K$-orbits in $\cN_1$ is parametrized by $\cP(N)$. When $N$ is even, the set of $\mathrm{O}(N)$-orbits in $\cN_1$  is parametrized by $\cP(N)$, moreover, each $\mathrm{O}(N)$-orbit remains one $K$-orbit if $\lambda$ has at least one odd part, and splits into two $K$-orbits otherwise. For $\lambda\in\cP(N)$, we write $\cO_\lambda$ for the corresponding nilpotent $K$-orbit in $\cN_1$ when $\lambda$ has at least one odd part, and write $\cO_{\lambda}^{\I}$ and $\cO_\lambda^{\II}$ for the corresponding two nilpotent $K$-orbits in $\cN_1$ when $\lambda$ has only even parts.  Suppose that $\lambda=(\lambda_1\geq\lambda_2\geq\cdots\geq\lambda_s>0)$. 
For $x\in\cO_\lambda$ (or $\cO_{\lambda}^{\omega}$, $\omega=\I,\II$), we have (see  \cite{S})
\beq\label{dim-centralizer}
\dim Z_K(x)=\sum_{i=1}^s(i-1)\lambda_i.
\eeq

Let $\La$ be a maximal abelian subspace of $\Lg_1$, which is also a Cartan subspace of $\Lg$. We have the ``little" Weyl group 
\beqn
W=N_K(\La)/Z_K(\La)=S_N.
\eeqn

\subsection{Equivariant fundamental group and its representations}\label{ssec-representations of tB}

As was discussed in \cite{CVX1}, the equivariant fundamental group 
\beqn
\pi_1^K(\Lg_1^{rs})\cong Z_K(\La)\rtimes B_N\cong(\bZ/2\bZ)^{N-1}\rtimes B_{N},
\eeqn
where $B_N$ is the braid group of $N$ strands and it acts on 
\beqn
Z_K(\La)\cong\{(i_1,\ldots,i_{N})\in (\bZ/2\bZ)^{N}\,|\,\sum_{k=1}^N i_k=0\}\cong (\bZ/2\bZ)^{N-1}
\eeqn
via the natural map $B_{N}\to S_{N}$. For simplicity we write 
\beqn
\text{$\widetilde{B}_{N}=(\bZ/2\bZ)^{N-1}\rtimes B_{N}$ and $I_{N}=(\bZ/2\bZ)^{N-1}$.}
\eeqn

It is easy to see that the action of $B_N$ on $I_{N}^\vee$ has $[N/2]+1$ orbits. We choose a set of representatives 
\beqn
\chi_m\in I_N^\vee,\ 0\leq m\leq[N/2],
\eeqn of the $B_N$-orbits as follows. Let $\tau_i'\in(\bZ/2\bZ)^N$ be the element with all entries 0 except the $i$-th position. Then $\{\tau_i=\tau_i'+\tau_{i+1}',\  i=1,\ldots,N-1\}$, is a set of generators for $I_N$. For $0\leq m\leq [N/2]$, we define a character $\chi_m$ as follows:
\beq\label{def of chim}
\chi_m(\tau_m)=-1\text{ and }\chi_m(\tau_i)=1\text{ for }i\neq m.
\eeq

For $\chi\in I_N^\vee$, we set
\beqn
B_\chi=\on{Stab}_{B_{N}}\chi.
\eeqn
Let $s_i$, $i=1,\ldots,N-1$, be the simple reflections in $W=S_N$. It is easy to check that 
\beq\label{stabilizer in sn}
\begin{gathered}
\text{$\on{Stab}_{S_N}(\chi_m)=\langle s_i,i\neq m\rangle\cong S_m\times S_{N-m}$ if $m\neq N/2$, and}\\
\text{$\on{Stab}_{S_N}(\chi_m)$ contains $S_m\times S_{m}$ as an index 2 normal subgroup if $m=N/2$.}
\end{gathered}
\eeq
Let us define 
\beqn
B_{m,N-m}=\text{ the inverse image of $S_m\times S_{N-m}\cong\langle s_i,i\neq m\rangle$ under the map $B_{N}\to S_{N}$}.
\eeqn
Then it follows from~\eqref{stabilizer in sn} that 
\beq\label{stabilizer in bn}
\begin{gathered}
\text{$B_{\chi_m}=B_{m,N-m}$  when $m\neq N/2$,}\\
\text{ and $B_{\chi_{m}}$ contains $B_{m,N-m}$ as an index 2 normal subgroup when $m=N/2$.}
\end{gathered}
\eeq

Let $\sigma_i,i=1,\ldots,N-1$, be the standard generators of $B_{N}$ which are lifts of the $s_i$'s under the map $B_N\to S_N$. Then $B_{m,N-m}$ is generated by $\sigma_i$, $i\neq m$, and $\sigma_m^2$. We have a natural quotient map 
\beq\label{quotient}
\bC[B_{m,N-m}]\twoheadrightarrow \cH_{m,-1}\otimes\cH_{N-m,-1}\cong\bC[B_{m,N-m}]/\langle(\sigma_i-1)^2,i\neq m,\sigma_m^2-1\rangle.
\eeq
Note that in the above formula the $\sigma_i$ corresponds to $-T_i$ of the Hecke algebra in~\eqref{hecke relations}. For us  the $\sigma_i$ are more natural generators since they arise from geometry as unipotent monodromy operators.

Let us write 
\beqn\cH_{m,-1}\otimes\cH_{N-m,-1}=\cH_{\chi_m,-1}.
\eeqn We consider a family of representations of $\widetilde{B}_{N}$ as follows.  For $0\leq m\leq [N/2]$, we define
\beq\label{induced module}
L_{\chi_m}\coloneqq\on{Ind}_{\bC[B_{m,N-m}]}^{\bC[B_{N}]}\cH_{\chi_m,-1}\cong \bC[B_{N}]\otimes_{\bC[B_{m,N-m}]}\cH_{\chi_m,-1}
\eeq
where in the tensor product $\bC[B_{m,N-m}]$ acts on $\cH_{\chi_m,-1}$ via the quotient map~\eqref{quotient} and on $\bC[B_{N}]$ by right multiplication. The module $L_{\chi_m}$ has a natural $\widetilde{B}_{N}$-action defined as follows. We let $B_{N}$ act on $L_{\chi_m}$ by left multiplication and we let $I_{N}$ act on $L_{\chi_m}$ via $a.(b\otimes v)=\left((b.\chi_m)(a)\right)( b\ot v)$ for $a\in I_{N}$, $b\in B_{N}$ and $v\in\cH_{\chi_m,-1}$. We view $L_{\chi_m}$ as a representation of the equivariant fundamental group $\widetilde{B}_{N}$ in this manner.

We will next identify the composition factors of the modules $L_{\chi_m}$.
Let $\mu^1\in\cP_2(m)$ and $\mu^2\in\cP_2(N-m)$, $m\in[0,[N/2]]$. Proceeding just as in the definition of $L_{\chi_m}$, one obtains the following  representation of $\widetilde{B}_{N}$:
\beq\label{induced module L}
V_{\mu^1,\mu^2}\coloneqq\on{Ind}_{\bC[B_{m,N-m}]}^{\bC[B_{N}]}(D_{\mu^1}\otimes D_{\mu^2})\cong \bC[B_{N}]\otimes_{\bC[B_{m,N-m}]}(D_{\mu^1}\otimes D_{\mu^2}).
\eeq
Using~\eqref{stabilizer in bn}, one readily checks that $V_{\mu^1,\mu^2}$ is an irreducible representation of $\widetilde{B}_{N}$ when $m\neq N/2$, or when $m=N/2$ and $\mu^1\neq\mu^2$. When $m=N/2$ and $\mu^1=\mu^2$, $V_{\mu^1,\mu^2}$ breaks into the direct sum of two non-isomorphic irreducible representations of $\widetilde{B}_N$, which we denote by $V_{\mu^1,\mu^2}^{\I}$ and $V_{\mu^1,\mu^2}^{\II}$, i.e., we have
\beq\label{two modules}
V_{\mu,\mu}\cong V_{\mu,\mu}^{\I}\oplus V_{\mu,\mu}^{\II}.
\eeq
Moreover, 
\beqn
\begin{gathered}
\text{when $m\neq N/2$, $V_{\mu^1,\mu^2}\cong V_{\nu^1,\nu^2}$ if and only if $(\mu^1,\mu^2)=(\nu^1,\nu^2)$;}\\
\text{ when $m=N/2$, $V_{\mu^1,\mu^2}\cong V_{\nu^1,\nu^2}$ if and only if}\\\text{ either $(\mu^1,\mu^2)=(\nu^1,\nu^2)$ or  $(\mu^1,\mu^2)=(\nu^2,\nu^1)$. }
\end{gathered}
\eeqn

As the $D_{\mu^1}\otimes D_{\mu^2}$ are the composition factors of $\cH_{\chi_m,-1}$ we conclude:

\begin{lemma} 
\label{composition factors}
The composition factors of $L_{\chi_m}$ consist of the  $V_{\mu^1,\mu^2}$, $\mu^1\neq\mu^2$, $\mu^1\in\cP_2(m)$, $\mu^2\in\cP_2(N-m)$, and when $N=2m$ we have two additional composition factors $V_{\mu,\mu}^{\I}$ and $V_{\mu,\mu}^{\II}$ for $\mu\in\cP_2(m)$.
\end{lemma}

\subsection{Parabolic induction functor and Fourier transform}Let $L$ be a $\theta$-stable Levi subgroup contained in a $\theta$-stable parabolic subgroup $P\subset G$. We write 
\beqn
\text{$\Ll=\on{Lie}L$, $\Lp=\on{Lie}P$, $L_K=L\cap K$, $P_K=P\cap K$, $\Ll_1=\Ll\cap\Lg_1$, $\Lp_1=\Lp\cap\Lg_1$.}
\eeqn
We will make use of the parabolic  induction functor 
\beqn\on{Ind}_{\fl_1\subset\mathfrak{p}_1}^{\Lg_1}:D_{L_K}(\fl_1)\to D_K(\Lg_1)
\eeqn defined in \cite{H,Lu3}. 
In the following we recall its definition. Let 
 \beqn
 \on{pr}:\Lp_1=\Ll_1\oplus(\Ln_P)_1\to\Ll_1
 \eeqn be the natural projection map, where $\Ln_P$ is the nilpotent radical of $\Lp$ and $(\Ln_P)_1=\Ln_P\cap\Lg_1$. Consider the diagram
\beq\label{induction diagram}
\xymatrix{\mathfrak{l}_1&\Lp_1\ar[l]_-{\on{pr}}&K\times\fp_1\ar[l]_-{p_1}\ar[r]^-{p_2}&K\times^{P_K}\fp_1\ar[r]^-{\check{\pi}}&\Lg_1},
\eeq
where $p_1$ and $p_2$ are natural projection maps and $\check{\pi}:(k,x)\mapsto \on{Ad}(k)(x)$. 

The maps in~\eqref{induction diagram} are $K\times P_K$-equivariant, where $K$ acts trivially on $\Ll_1$ and $\Lp_1$, by left multiplication on the $K$-factor on $K\times\Lp_1$ and on $K\times^{P_K}\Lp_1$, and by adjoint action on $\Lg_1$, and $P_K$ acts on $\Ll_1$ by $a.l=\on{pr}(\on{Ad}a(l))$, by adjoint action on $\Lp_1$, by $a.(k,p)=(ka^{-1},\on{Ad}a(p))$ on $K\times\Lp_1$, trivially on $K\times^{P_K}\Lp_1$ and $\Lg_1$. 

Let $A$ be a complex in $D_{L_K}(\fl_1)$. Then $(\on{pr}\circ p_1)^*A\cong p_2^*A'$ for a well-defined complex $A'$ in  $D_K(K\times^{P_K}\fp_1)$. Define
\beqn
\on{{I}nd}_{\fl_1\subset\mathfrak{p}_1}^{\Lg_1}A=\check{\pi}_{!}A'[\dim P-\dim L].
\eeqn
Let $\fF:D_K(\Lg_1)\to D_K(\Lg_1)$ be the Fourier transform functor (we also use the same notation $\fF$ for the functor defined for $(L_K,\Ll_1)$). Here we have identified $\Lg_1$ with $\Lg_1^*$ via a $K$-invariant non-degenerate bilinear form on $\Lg_1$. It is shown in \cite{H,Lu3} that the induction functor commutes with Fourier transform, i.e.,
\beq\label{commutativity}
\fF( \on{Ind}_{\fl_1\subset\mathfrak{p}_1}^{\Lg_1}A)\cong\on{Ind}_{\fl_1\subset\mathfrak{p}_1}^{\Lg_1}(\fF(A)).
\eeq

\section{Maximal $\theta$-stable parabolic subgroups and parabolic induction}\label{sec-parabolic}

In this section, we study the parabolic induction functor  with respect to a chosen family of $L^m\subset P^m$, $1\leq m<N/2$, and two more pairs $L^{n,\omega}\subset P^{n,\omega}$, $\omega=\I,\II$, if $N=2n$, where $P^m$ (resp. $P^{n,\omega}$) is a maximal  $\theta$-stable parabolic subgroup and $L^m$ (resp. $L^{n,\omega}$) is a $\theta$-stable Levi subgroup of $P^m$ (resp. $P^{n,\omega}$) defined as follows. 

Fix a basis $\{e_i, 1\leq i\leq N\}$ of $V$ such that $\la e_i,e_j\ra=\delta_{i+j,N+1}$. For $1\leq m< N/2$, we define $P^m$ to be the parabolic subgroup of $G$ that stabilizes the flag
\beqn
0\subset V_{m}^0\subset V_m^{0\p}\subset \bC^{N},
\eeqn 
where $V_m^0=\on{span}\{e_i,1\leq i\leq m\}$.
We define $L^m$ to be the $\theta$-stable Levi subgroup  of $P^m$ which consists of diagonal block matrices of sizes $m,N-2m,m$. When $N=2n$, for $\omega=\I,\II$, we define $P^{n,\omega}$ to be the parabolic subgroup of $G$ that stabilizes the flag
\beqn
0\subset V_{n}^\omega\subset V_n^{\omega\p}\subset \bC^{2n},
\eeqn 
where $V_n^\I=\on{span}\{e_i,1\leq i\leq n\}$ and $V_n^{\II}=\on{span}\{e_i,1\leq i\leq n-1,e_{n+1}\}$. Let  $L^{n,\omega}$ be a $\theta$-stable Levi subgroup of $P^{n,\omega}$. According to \cite{BH}, every maximal $\theta$-stable parabolic subgroup of $G$ is $K$-conjugate to one of the above form. 

Let $\fp^m=\on{Lie}\,P^m$, $\fp^m_1=\fp^m\cap\Lg_1$, and $(\Ln_{P^m})_1=\Ln_{P^m}\cap\Lg_1$, where $\Ln_{P^m}$ is the nilpotent radical of $\Lp^m$, etc.  

\begin{proposition}\label{small}We have:
\begin{enumerate}
\item The map
$${\pi}_m^N:K\times^{P^{m}_K}(\Ln_{P^m})_1\to\cN_1,\ (k,x)\mapsto\on{Ad}k(x)$$
is a small map onto its image, generically one-to-one.
\item The map \beqn
\check{\pi}_m^N:K\times^{P^m_K}\fp^m_1\to\Lg_1,\ (k,x)\mapsto\on{Ad}k(x).
\eeqn is a small map onto its image, generically one-to-one.
\end{enumerate}
The same holds for the two maps ${\pi}_n^{2n,\omega}$ and $\check{\pi}_n^{2n,\omega}$ defined using $P^{n,\omega}$, $\omega=\I,\II$.
 \end{proposition}
 
We define 
\beq\label{definition of g1m}
\Lg_1^m=\on{Im}\check{\pi}_m^N, \ 1\leq m<N/2,\ \Lg_1^{n,\omega}=\on{Im}\check{\pi}_n^{2n,\omega},\,\omega=\mathrm{I},\mathrm{II}.
\eeq
For $m<N/2$, $\Lg_1^m$ consists of elements in $\Lg_1$ with eigenvalues $a_1,a_1,\ldots,$ $a_m,a_m,$ $a_j, j\in[2m+1,N]$, where $\sum _{k=1}^m2a_k+\sum_{j=2m+1}^{N}a_j=0$. Let 
\beqn
\begin{gathered}
\text{$Y^r_m=\{x\in\Lg_1^{reg}\,|\, x$ has eigenvalues $a_1,a_1,\ldots, a_m,a_m, a_j, j\in[2m+1,N]$,}\\\hspace{.5in}\text{ where $a_i\neq a_j$ for $i\neq j$\}.}
\end{gathered}
\eeqn One checks readily that $\overline{Y^r_m}=\Lg_1^m$. 

Consider the case $m=N/2=n$. For $\omega=\I,\II$, let
\beqn
\begin{gathered}
\text{$Y^{r,\omega}_n=\{x\in\Lg_1^{reg}\,|\, x$ has eigenvalues $a_1,a_1,\ldots, a_n,a_n$, where $a_i\neq a_j$ for $i\neq j$,}\\\hspace{.5in}\text{  and the nilpotent part of $x$ lies in the orbit $\cO_{2^n}^{\omega}$\},}
\end{gathered}
\eeqn
where $\cO_{2^n}^\omega$ is the nilpotent orbit given by the partition $2^m$ and defined by the equation $\on{Im}\pi_n^{2n,\omega}=\bar\cO_{2^n}^\omega$. 
Then $Y^{r,\omega}_n$ is an open dense subset in $\Lg_1^{n,\omega}$.

Let $(\fp^m_1)^r=\fp^m_1\cap Y_m^r$ and $(\Ll^m_1)^{rs}=\Ll_1^m\cap(\Ll^m)^{rs}$. 
\begin{proposition}\label{prop-induction2}
(1) Suppose that $1\leq m<N/2$. There is a natural surjective map 
\beq\label{map of fd}
\pi_1^K(Y_m^r)\twoheadrightarrow\pi_1^{L^m_K}\left((\Ll^m_1)^{rs}\right)\cong B_m\times\widetilde{B}_{N-2m}
\eeq
such that for   an $L^m_K$-equivariant  local system $\cT$ on $(\Ll^m_1)^{rs}$ associated to a $\pi_1^{L^m_K}((\Ll_1^m)^{rs})$-representation $E$, we have
\beqn
\on{Ind}_{\Ll_1^m\subset\Lp_1^m}^{\Lg_1}\IC(\Ll_1^m,\cT)\cong\IC(\Lg_1^m,\cT'),
\eeqn
where $\cT'$ is the $K$-equivariant local system on $Y_m^r$  associated to the representation of $\pi_1^K(Y_m^r)$  which is obtained from $E$ by pull-back under the map~\eqref{map of fd}.

(2) We have a natural surjective map 
\beq\label{map of fd2}
\pi_1^K\left(Y_n^{r,\omega}\right)\twoheadrightarrow\pi_1^{L^{n,\omega}_K}\left(\left(\Ll^{n,\omega}_1\right)^{rs}\right)\cong B_n,\ \omega=\I,\II,
\eeq
such that for an  $L^{n,\omega}_K$-equivariant local system $\cT$ on $(\Ll^{n,\omega}_1)^{rs}$ associated  to a $\pi_1^{L^{n,\omega}_K}((L^{n,\omega}_1)^{rs})$-representation $E$, we have 
\beqn
\on{Ind}_{\Ll^{n,\omega}_1\subset\Lp^{n,\omega}_1}^{\Lg_1}\IC(\Ll^{n,\omega}_1,\cT)\cong\IC\left(\Lg_1^{n,\omega},\cT'\right),
\eeqn
where $\cT'$ is the $K$-equivariant local system on $Y_n^{r,\omega}$ associated to the representation of  $\pi_1^K(Y_n^{r,\omega})$ which is obtained from $E$ by pull-back under the map~\eqref{map of fd2}.
\end{proposition}

\subsection{Proof of Proposition~\ref{small}}
We begin with the proof of (1). 
Consider the following projection
\beqn
\tau_m^N:\{(x,0\subset V_m\subset V_m^\p\subset V=\bC^{N})\,|\,x\in\Lg_1,\ xV_m=0,\ xV_m^\p\subset V_{m}\}\to\cN_1.
\eeqn
When $m\neq N/2$, the map $\tau_m^N$ can be identified with the map $\pi_m^N$. When $N=2m$, the image of the map $\tau_m^{2m}$ has two irreducible components, i.e., closures of the two orbits $\cO_{2^m}^\I$ and $\cO_{2^m}^{\II}$. The two maps ${\pi}_m^{N,\I}$ and ${\pi}_m^{N,\II}$ can be identified with the map $\tau_m^{2m}$ restricted to $\overline{(\tau_m^{2m})^{-1}(\cO_{2^m}^\I)}$ and $\overline{(\tau_m^{2m})^{-1}(\cO_{2^m}^\II)}$ respectively.
Thus it suffices to show that
\beq\label{smallness of pi}
\text{ the map ${\tau}_m^N$ is small over its image and generically one-to-one}.
\eeq
When $m\neq N/2$, one can check that the image of $\tau_m^N$ is as follows
\beqn
\on{Im}\tau_m^N=\bar\cO_{3^{m}1^{N-3m}}\text{ if }N\geq 3m,\ \ \on{Im}\tau_m^N=\bar\cO_{3^{N-2m}2^{3m-N}}\text{ if }N<3m.
\eeqn
Assume that $N\geq 3m$ and $x\in\cO_{3^{m}1^{N-3m}}$. Then $(\tau_m^N)^{-1}(x)$ consists of the flag $0\subset V_m\subset V_m^\p\subset V$, where $V_m=\on{Im}x^2$. Assume that $N<3m$ and $x\in\cO_{3^{N-2m}2^{3m-N}}$. Then $(\tau_m^N)^{-1}(x)$ consists of the flag $0\subset V_m\subset V_m^\p\subset V$, where $V_m=\ker x$. This proves that $\tau_m^N$ is generically one-to-one.

Let $x\in\cO_{3^i2^j1^{N-3i-2j}}\subset\on{Im}\tau_m^N$. We assume that $3^i2^j1^{N-3i-2j}\neq 3^m1^{N-3m}$ if $N\geq 3m$, and $3^i2^j1^{N-3i-2j}\neq 3^{N-2m}2^{3m-N}$ if $N< 3m$.  It suffices to show that
\beqn
\dim(\tau_m^N)^{-1}(x)<\on{codim}_{\on{Im}\tau_m^N}\cO_{3^i2^j1^{N-3i-2j}}/2.
\eeqn  
Let $x_0\in\cO_{2^j1^{N-3i-2j}}\subset\on{Im}\tau_{m-i}^{N-3i}$. (Note that $\tau_{m-i}^{N-3i}$ is defined since $m-i\leq (N-3i)/2$.) One checks readily that (using~\eqref{dim-centralizer} for the second equality)
\beqn
(\tau_m^{N})^{-1}(x)\cong(\tau_{m-i}^{N-3i})^{-1}(x_0)
\text{ and }
\on{codim}_{\on{Im}\tau_m^N}\cO_{3^i2^j1^{N-3i-2j}}=\on{codim}_{\on{Im}\tau_{m-i}^{N-3i}}\cO_{2^j1^{N-3i-2j}}.
\eeqn
Thus it suffices to show that
\beqn
\dim(\tau_{m-i}^{N-3i})^{-1}(x_0)<\on{codim}_{\on{Im}\tau_{m-i}^{N-3i}}\cO_{2^j1^{N-3i-2j}}/2.
\eeqn
Let us write 
\beqn
\Omega_{m,j}^N=(\tau_m^N)^{-1}(\zeta_j)\text{ for }\zeta_j\in\cO_{2^j1^{N-2j}}\subset \on{Im}\tau_m^N
\eeqn
\beqn
 \text{ and }a_{m,j}^N=\on{codim}_{\on{Im}\pi_{m}^{N}}\cO_{2^j1^{N-2j}}=m(2N-3m)-j(N-j).
 \eeqn
  To prove that the map $\tau_m^N$ is small, we are reduced to proving that
\beq\label{dimest}
\dim\Omega_{m,j}^N<\frac{a_{m,j}^N}{2}.
\eeq
To prove this we recall the partitioning of  $\Omega_{m,j}^N$ into $\Omega_{m,j}^{N,k}$ given in \cite[Section2]{CVX2} as follows:
$$
\Omega_{m,j}^{N,k}=\{(0\subset V_m\subset V_m^\p\subset V=\bC^{N})\,|\,\dim(V_m\cap \zeta_jV)=k\}.
$$
We have
\beqn \Omega_{m,j}^{N,k}\neq\emptyset\Leftrightarrow\max\{m+j-N/2,j/2\}\leq k\leq\min\{j,m\}.
\eeqn
Recall that we have a surjective map $\Omega_{m,j}^{N,k}\to\on{OGr}(j-k,j)\times\on{OGr}(m-k,N-2j)$ with fibers being affine spaces $\mathbb{A}^{(m-k)(j-k)}$. We have
\beqn
\dim \overline{\Omega_{m,j}^{N,k}}=-2k^2+(-N+3j+2m+1)\,k+mN-mj-\frac{j^2+3m^2+j+m}{2}.
\eeqn
One checks that
\beqn
\begin{gathered}
\text{if }j\geq N-2m,\ \dim\overline{ \Omega_{m,j}^{N,k}} \text{ is maximal when }k=m+j-[\frac{N}{2}],\\
\text{if }j<N-2m,\  \dim\overline{ \Omega_{m,j}^{N,k}} \text{ is maximal when }k=[\frac{j+1}{2}].
\end{gathered}
\eeqn
Thus a direct calculation shows that
\beqn
\dim(\pi_m^{N})^{-1}(\zeta_j)=\left\{\begin{array}{cc}\frac{a_{m,j}^N}{2}+\frac{j+m-N}{2}&\text{if } j\geq N-2m\text{ and }N\text{ even},\text{ or }j< N-2m\text{ and }j\text{ odd}\\
\frac{a_{m,j}^N}{2}-\frac{m}{2}&\text{if }j\geq N-2m\text{ and }N\text{ odd}, \text{ or }j< N-2m\text{ and }j \text{ even}.
\end{array}\right.
\eeqn
This proves \eqref{dimest} (note that $m+j<N$). The proof of \eqref{smallness of pi} is complete. This finishes the proof of the claim (1) in the proposition.
 
 It then follows that  we have
 \beq\label{image1}
(\pi_m^N)_*\bC[-]\cong\on{IC}(\cO_\lambda,\bC),\ (\text{resp. }((\pi^N_{N/2})^\omega)_*\bC[-]\cong\on{IC}(\cO_{\lambda}^\omega,\bC),\ \omega=\I,\II,)
 \eeq
 where \beqn
\lambda={3^{m}1^{N-3m}}\text{ if }N\geq 3m,\ \ \lambda={3^{N-2m}2^{3m-N}}\text{ if }N<3m.
\eeqn
Note that $K\times^{P^m_K}\Lp^m_1$ is the orthogonal complement of $K\times^{P^m_K}(\Ln_{P^m})_1$ in the trivial bundle $K\times\Lg_1$ over $K/P^m_K$. 
By the functoriality of Fourier transform, we have that
\beq\label{FT in prop}
\mathfrak{F}\left((\pi_m^N)_*\bC[-]\right)\cong(\check{\pi}_m^N)_*\bC[-].
\eeq
Since Fourier transform sends simple perverse sheaves to simple perverse sheaves, we can conclude from~\eqref{image1} and~\eqref{FT in prop} that 
\beqn
(\check{\pi}_m^N)_*\bC[-]\cong \on{IC}(\on{Im}\check{\pi}_m^N,\bC).
\eeqn
This proves the claim (2) of the proposition. The argument for $(\check{\pi}_n^{2n })^{\omega}$, $\omega=\I,\II$, is the same.
The proof of the proposition is complete.

\subsection{Proof of Proposition~\ref{prop-induction2}}

Note that we have that 
\beq\label{LK and l1}
\text{$L^m_K\cong \GL(m)\times \SO(N-2m)$ and $(\Ll^m)_1\cong\mathfrak{gl}(m)\oplus \mathfrak{sl}(N-2m)_1$.}
 \eeq
To ease  notations, let us write now that $L=L^m$, $P=P^m$, and $\check{\pi}=\check{\pi}_m^N$ etc.

We first show that
\beq\label{claim-one-to-one}
\begin{gathered}
\text{The map $\check{\pi}$ (resp. $\check{\pi}_n^\omega$), when restricted to $\check{\pi}^{-1}(Y^r)$ (resp. $\check{\pi}^{-1}(Y_n^{r,\omega})$),}\\
\text{ is one-to-one. }
\end{gathered}
\eeq
Each element in $Y^r$ is $K$-conjugate to an element $x_0\in\fp_1$ (see \cite[Theorem 7]{KR}), where
\beq\label{normal form}
\begin{gathered}
x_0e_i=a_ie_i,\ x_0e_{N+1-i}=e_i+a_ie_{N+1-i}\text{ for }i\in[1,m],\\
x_0e_j=b_je_j+c_je_{N+1-j},\ x_0e_{N+1-j}=c_je_j+b_je_{N+1-j}\text{ for }j\in[m+1,[N/2]]\\
x_0e_{(N+1)/2}=b_{(N+1)/2}e_{(N+1)/2} \text{ if $N$ is odd}
\end{gathered}
\eeq
and the numbers $a_i,i=1,\ldots,m,b_j+c_j,b_j-c_j,j=m+1,\ldots, [N/2],b_{(N+1)/2}$ are distinct.

Thus it suffices to show that $\check{\pi}^{-1}(x_0)$ consists of one point. Note that $\check{\pi}^{-1}(x_0)$ consists of $x_0$-stable $m$-dimensional isotropic subspaces of $V$.  Assume that $U_m\in\check{\pi}^{-1}(x_0)$. Then $U_m$ is a direct sum of generalized eigenspaces of $x_0$. The generalized eigenspace of $x_0$ with eigenvalue $a_i$ is $\on{span}\{e_i,e_{N+1-i}\}$, $i=1,\ldots,m$, and the eigenspace of $x_0$ with eigenvalue $b_j+c_j$, $b_j-c_j$, $b_{(N+1)/2}$ is $\on{span}\{e_j+e_{N+1-j}\}$, $\on{span}\{e_j+e_{N+1-j}\}$, $\on{span}\{e_{(N+1)/2}\}$, respectively. Since $U_m$ is isotropic, $U_m$ has to equal $\on{span}\{e_i,i=1,\ldots,m\}$. This proves~\eqref{claim-one-to-one} for $\check{\pi}_m$, $m<N/2$. The proof for $\check{\pi}_n^\omega$ is entirely similar and omitted.

Now we show that
\beq\label{image of p}
\text{The image of $\fp^r_1$ under the map $\on{pr}:\Lp_1\to\Ll_1$ is $\Ll_1^{rs}$.}
\eeq
Let $x\in \Lp_1^r$. By the above proof of~\eqref{claim-one-to-one} we can assume that $\on{Ad}(k)x=x_0$ for some $k\in K$,  where  $x_0$ is as in~\eqref{normal form}. Thus $(k,x)\in\check{\pi}^{-1}(x_0)$. It follows from~\eqref{claim-one-to-one} that $(k,x)=(1,x_0)\in K\times^{P_K}\Lp_1$. Hence $k\in P_K$. Assume that $k=lu$ where $l\in L_K$ and $u\in U_K=U\cap K$ ($U$ is the unipotent radical of $P$). Then we have $\on{pr}(x)=\on{pr}(\on{Ad}(u^{-1}l^{-1})x_0)=\on{pr}(\on{Ad}(l^{-1})x_0)=\on{Ad}(l^{-1})\on{pr}(x_0)$. It is clear that $\on{pr}(x_0)\in\Ll^{rs}$. Thus~\eqref{image of p} follows.

By~\eqref{claim-one-to-one} and~\eqref{image of p}, we have the following diagram, when restricting~\eqref{induction diagram} to $Y^r$,
\beqn
\xymatrix{\mathfrak{l}_1^{rs}&\Lp_1^r\ar[l]_-{\on{pr}}&K\times\fp_1^r\ar[l]_-{p_1}\ar[r]^-{p_2}&K\times^{P_K}\fp_1^r\ar[r]^-{\check{\pi}}&Y^r}.
\eeqn
Using~\eqref{claim-one-to-one}, we see that 
\beqn
\pi_1^K(Y^r)\cong\pi_1^{K\times P_K}(Y^r)\cong\pi_1^{K\times P_K}(K\times^{P_K}\Lp_1^r)\cong\pi_1^{K\times P_K}(K\times\Lp_1^r)\cong\pi_1^{P_K}(\Lp_1^r).
\eeqn
Finally, the canonical map $\pi_1^{P_K}(\Lp_1^r)\rightarrow\pi_1^{P_K}(\Ll_1^{rs})\cong\pi_1^{L_K}(\Ll_1^{rs})$ is surjective. We see this as follows. First, the canonical map above can be identified with the canonical map  $\pi_1^{P_K}(\Lp_1^r)\rightarrow\pi_1^{P_K}(\on{pr}^{-1}(\Ll_1^{rs}))$. Now, because $\Lp_1^r$ is an open subset in $\on{pr}^{-1}(\Ll_1^{rs})$, which is smooth, the map $\pi_1(\Lp_1^r)\rightarrow\pi_1(\on{pr}^{-1}(\Ll_1^{rs}))$ is a surjection. To conclude that this property persists when we pass to the equivariant fundamental group it suffices to remark that the equivariant fundamental group is always a quotient of the ordinary fundamental group as long as the group is connected. 
We now conclude the argument making use of Proposition \ref{small}.

\section{Fourier transform of nilpotent orbital complexes for \texorpdfstring{$(\SL(N), \SO(N))$}{lg}}\label{sec-FT}
Consider the symmetric pair $(G,K)=(\SL(N), \SO(N))$. Let us write
$
\cA_{N}$ for the set of all simple $K$-equivariant perverse sheaves on $\cN_1$ (up to isomorphism), that is, the set of IC complexes $\IC(\cO,\cE)$, where $\cO$ is a $K$-orbit in $\cN_1$ and $\cE$ is an irreducible $K$-equivariant local system on $\cO$ (up to isomorphism). An IC complex in $\cA_{N}$ is called a nilpotent orbital complex.

Let $n=[N/2]$. We set
\beqn
\begin{gathered}
\Sigma_{N} = \{(\nu;\mu^1,\mu^2)\mid 0\leq m \leq n, \ \nu\in\cP(m)\\
 0\leq k \leq n-m , \ \mu^1\in\cP_2(k), \mu^2\in\cP_2(N-2m-k)  \}\,.
\end{gathered}
\eeqn
In the case when $N$ is even, we identify the triple $(\nu;\mu^1,\mu^2)$ with $(\nu;\mu^2,\mu^1)$ if $|\mu^1|=|\mu^2|$ and $\mu^1\neq\mu^2$, and the triples $(\nu;\mu,\mu)$ attain two labels $\I$ and $\II$.

Given a triple $(\nu;\mu^1,\mu^2)\in\Sigma_{N}$ (resp. $(\nu;\mu,\mu)^{\omega}\in\Sigma_N$, $\omega=\I,\II$), where $|\nu|=m<N/2$,  we define  an irreducible $K$-equivariant local system $\cT(\nu;\mu^1,\mu^2)$ (resp. $\cT(\nu;\mu,\mu)^{\omega}$) on $Y^r_m$ (here we write $Y^r_0=\Lg_1^{rs}$) as follows. We obtain  a map
\beqn
\tau:\pi_1^K(Y_m^r)\to B_m\times\widetilde{B}_{N-2m}\to  S_m\times \widetilde{B}_{N-2m}
\eeqn
by composing the map in~\eqref{map of fd} with the natural map $B_m\times\widetilde{B}_{N-2m}\to S_m\times \widetilde{B}_{N-2m}$. 
Note that the map $\tau$ is surjective. Then $\cT(\nu;\mu^1,\mu^2)$ (resp. $\cT(\nu;\mu,\mu)^{\omega}$) is the irreducible local syste associated to the irreducible representation of $\pi_1^K(Y^r_m)$ given by pulling back the irreducible representation $\rho_\nu\boxtimes V_{\mu^1,\mu^2}$ (resp. $\rho_\nu\boxtimes V_{\mu,\mu}^\omega$) via the map $\tau$; here $\rho_\nu\in S_m^\vee$ is the irreducible representation of $S_m$ corresponding to $\nu\in\cP(m)$  and $V_{\mu^1,\mu^2}$ (resp. $V_{\mu,\mu}^\omega$) is the irreducible representation of $\widetilde{B}_{N-2m}$ defined in~\eqref{induced module} (resp. \eqref{two modules}).

Assume now that $N=2n$.  Given a triple $(\nu;\emptyset,\emptyset)^\omega\in\Sigma_{N}$, $\omega=\I,\II$, we define the irreducible $K$-equivariant local system $\cT(\nu;\emptyset,\emptyset)^{\omega}$ on $(Y_n^r)^\omega$ as the local system associated to the representation of $\pi_1^K((Y_n^r)^\omega)$ obtained by pulling back the representation $\rho_\nu\in S_n^\vee$ corresponding to $\nu\in\cP(n)$ under that map
\beqn
\pi_1^K((Y_n^r)^\omega)\twoheadrightarrow B_n \twoheadrightarrow S_n\,.
\eeqn

Now we are ready to formulate our main result:
\begin{theorem}\label{main theorem}
The Fourier transform $\fF:\on{Perv}_K(\Lg_1)\to\on{Perv}_K(\Lg_1)$ induces a bijection
\beqn
\begin{gathered}
\fF:\cA_{N}\xrightarrow{\sim}\left\{\IC\left(\Lg_1^m,\cT(\nu;\mu^1,\mu^2)\right)|\left(\nu;\mu^1,\mu^2\right)\in\Sigma_{N},\mu^1\neq\mu^2,\ |\nu|=m<N/2\right\}\\
\qquad\qquad\cup \left\{\IC\left(\Lg_1^m,\cT(\nu;\mu,\mu)^\omega\right)|\left(\nu;\mu,\mu\right)^\omega\in\Sigma_{N},\omega=\I,\II,\ |\nu|=m<N/2\right\}\text{ (if $N$ is even)},\\
\qquad\qquad\cup \left\{\IC\left(\Lg_1^{n,\omega},\cT(\nu;\emptyset,\emptyset)^\omega\right)|\left(\nu;\emptyset,\emptyset\right)^\omega\in\Sigma_{N},\omega=\I,\II,\ |\nu|=n=N/2\right\}\text{ (if $N$ is even)},
\end{gathered}
\eeqn
where $\Lg_1^0=\Lg_1$, $\Lg^m_1$ and $\Lg^{n,\omega}_1$ are defined in~\eqref{definition of g1m}.
\end{theorem}
\subsection{Proof of Theorem~\ref{main theorem}}Let $p(k)$ denote the number of partitions of $k$ and let $q(k)$ denote the number of 2-regular partitions of $k$. We write $p(0)=q(0)=1$. Let us define 
\beq\label{def of d(k)}
d(k)=\sum_{s=0}^kq(s)q(2k+1-s),
\eeq
\beq\label{def of e(k)}
e(k)=\sum_{s=0}^{k-1}q(s)\,q(2k-s)+\frac{q(k)^2+3q(k)}{2}.
\eeq
 \begin{lemma}\label{lemma-count}
We have 
\beq\label{odd number}
|\cA_{2n+1}|=\sum_{k=0}^{n}p(n-k)d(k)=|\Sigma_{2n+1}|
\eeq
\beq\label{even number}
|\cA_{2n}|=\sum_{k=0}^{n}p(n-k)e(k)=|\Sigma_{2n}|.
\eeq
\end{lemma}
\begin{proof}
Note that
\beq\label{partition eq1}
\sum_{k\geq 0}p\left(k\right)x^k=\prod_{s\geq 1}\frac{1}{1-x^s}\text{ and } \sum_{k\geq 0}q(k)x^k=\prod_{s\geq 1}\left(1+x^s\right).
\eeq
 Let $p(l,k)$ denote the number of partitions of $l$ into (not necessarily distinct) parts of exactly $k$ different sizes. We have (see for example \cite{W})
\beq\label{partition-1}
\sum_{l,k\geq 0}p(l,k)x^ly^k=\prod_{s\geq 1}\left(1+\frac{yx^s}{1-x^s}\right).
\eeq

Assume first that $N=2n+1$. Note that if $\lambda$ is a partition of $N$ with parts of $k$ different sizes, then the component group $A_K(x)$ of the centralizer $Z_K(x)$ for $x\in\cO_\lambda$ is $(\bZ/2\bZ)^{k-1}$. Thus there are $2^{k-1}$ irreducible $K$-equivariant local systems on $\cO_\lambda$ (up to isomorphism). Hence using~\eqref{partition-1}, we see that
\beqn
|\cA_{2n+1}|=\sum_{k\geq0}p\left(2n+1,k\right)2^{k-1}=\text{Coefficient of $x^{2n+1}$ in }\frac{1}{2}\prod_{s\geq 1}\left(\frac{1+x^s}{1-x^s}\right).
\eeqn
Using~\eqref{partition eq1}, we see that 
\beq\label{partition eq2}
\prod_{s\geq 1}\left(\frac{1+x^s}{1-x^s}\right)=\left(\sum_{k\geq 0}p\left(k\right)x^{2k}\right)\left(\sum_{k\geq 0}q\left(k\right)x^k\right)^2.
\eeq
It then follows that $|\cA_{2n+1}|$ is the desired number. The fact that $|\Sigma_{2n+1}|$ equals the same number is clear from the definition. Thus~\eqref{odd number} holds.

Assume now that $N=2n$. Suppose that $\lambda$ is a partition of $N$ with parts of exactly $k$ different sizes. If $\lambda$ has at least one odd part, then there are $2^{k-1}$ irreducible $K$-equivariant local systems on $\cO_\lambda$ (up to isomorphism). If $\lambda$ has only even parts, then there are $2^{k}$ irreducible $K$-equivariant local systems on each $\cO_\lambda^\omega$ (up to isomorphism), $\omega=\I,\II$.

Thus we have that
\beqn
\begin{gathered}
|\cA_{2n}|=\sum_{k\geq0}p(2n,k)\,2^{k-1}+\sum_{k\geq 0}p(n,k)\,3\cdot2^{k-1}\\
=\text{Coefficient of $x^{2n}$ in }\frac{1}{2}\prod_{s\geq 1}\left(\frac{1+x^s}{1-x^s}\right)+\text{Coefficient of $x^{n}$ in }\frac{3}{2}\prod_{s\geq 1}\left(\frac{1+x^s}{1-x^s}\right)\\
=\frac{1}{2}\left(\sum_{k=0}^{n}p(n-k)\left(2\sum_{s=0}^{k-1}q(s)q(2k-s)+q(k)^2\right)\right)+\frac{3}{2}\sum_{k=0}^np(n-k)q(k)=\sum_{k=0}^{n}p(n-k)e(k).
\end{gathered}
\eeqn
Here we have used~\eqref{partition eq2} and the following equation
\beqn
\prod_{s\geq 1}\left(\frac{1+x^s}{1-x^s}\right)=\left(\sum_{k\geq 0}p\left(k\right)x^{k}\right)\left(\sum_{k\geq 0}q\left(k\right)x^k\right).
\eeqn Again the fact that $|\Sigma_{2n}|$ equals the desired number is clear from the definition. 
\end{proof}

Note that the IC sheaves appearing on the right hand side of the Fourier transform map $\fF$ in Theorem~\ref{main theorem} are pairwise non-isomorphic. Thus, in view of Lemma~\ref{lemma-count}, Theorem~\ref{main theorem} follows from:\begin{proposition}
Let $({\nu;\mu^1,\mu^2})\in\Sigma_{N}$ (resp. $(\nu;\mu,\mu)^\omega\in\Sigma_N$, $\omega=\I,\II$) and write $m=|\nu|$. 
The Fourier transform of $\IC(\Lg_1^m,\cT({\nu;\mu^1,\mu^2}))$ (resp. $\IC(\Lg_1^m,\cT(\nu;\mu,\mu)^\omega)$, $\IC(\Lg_1^{n,\omega},\cT(\nu;\emptyset,\emptyset)^{\omega})$) is supported on a $K$-orbit in $\cN_1$.
\end{proposition}
\begin{proof}
Let $n=[N/2]$. We begin the proof by showing that for ${(\emptyset;\mu^1,\mu^2)}\in\Sigma_{N}$ (resp. $\left({\emptyset;\mu,\mu}\right)^\omega\in\Sigma_{N}$,  $\omega=\I,\II$)
\beq\label{FT full support}
\begin{gathered}
\text{The Fourier transform of $\IC\left(\Lg_1,\cT\left({\emptyset;\mu^1,\mu^2}\right)\right)$ $\left(\text{resp. }\IC\left(\Lg_1,\cT\left({\emptyset;\mu,\mu}\right)^\omega\right)\right)$}\\\text{ is supported on a $K$-orbit in $\cN_1$.}
\end{gathered}
\eeq
Recall that $\cT({\emptyset;\mu^1,\mu^2}))$ (resp. $\cT\left({\emptyset;\mu,\mu}\right)^\omega$) is the irreducible $K$-equivariant local system on $\Lg_1^{rs}$ corresponding to $V_{\mu^1,\mu^2}$ (resp. $V_{\mu,\mu}^\omega$).

We will now appeal to the nearby cycle construction in~\cite{GVX}.
Let us recall the characters $\chi_m\in I_{N}^\vee$, $0\leq m\leq n$ of~\eqref{def of chim}. In~\cite{GVX} we apply a nearby cycle construction to local systems associated to the $\chi_m$ and obtain a $K$-equivariant perverse sheaf $\cP_{\chi_m}$ on the nilpotent cone $\cN_1$. More precisely, for each character $\chi_m$, let us write 
$
W_{\chi_m}=\on{Stab}_{W}(\chi_m).
$
Consider the following base-change diagram of the adjoint quotient map
\beq
\label{nearby cycle diagram}
\begin{CD}
\Lg_{1,\chi_m} @>>> \Lg_1
\\
@V{f_{\chi_m}}VV    @VV{f}V
\\
\La/W_{\chi_m}@>>> \La/W\,.
\end{CD}
\eeq
Let us write $\Lg_{1,\chi_m}^{rs}$ for the base change of the regular semisimple locus $\Lg_1^{rs}$. Denote by $\cF_{\chi_m}$  the $K$-equivariant local system  on  $\Lg_{1,\chi_m}^{rs}$ corresponding to the representation of $\pi_1^{K}(\Lg_{1,\chi_m}^{rs})=I_N\rtimes B_{\chi_m}$ where $I_N$ acts via the character $\chi_m$ and $B_{\chi_m}$ acts trivially (recall that $B_{\chi_m}=\on{Stab}_{B_N}(\chi_m)$). We form the nearby cycle sheaf $\cP_{\chi_m}=\psi_{f_{\chi_m}}\cF_{\chi_m}$, appropriately shifted, so that $\cP_{\chi_m}\in\on{Perv}_K(\cN_1)$. 

Applying~\cite[Theorem 3.2]{GVX}, we obtain that
\beqn
\fF(\cP_{\chi_m})=\on{IC}(\Lg_1,\cM_{\chi_m}),
\eeqn
where $\cM_{\chi_m}$ is the $K$-equvariant local system on $\Lg_1^{rs}$ corresponding to the $\widetilde{B}_N$-representation
\beqn
M_{\chi_m}=\bC[\widetilde{B}_N]\otimes_{\bC[\widetilde{B}_{\chi_m}^0]}(\bC_{\chi_m}\otimes\cH_{W_{\chi_m}^0}).
\eeqn
Let us explain the notations in the above formula in our setting. Let $W_{\chi_m}^0$ be the Coexter subgroup of $W$ generated by $s_\alpha$ with $\chi_m(\check\alpha(-1))=1$, $\alpha\in\Phi(\Lg,\La)$, where $\Phi(\Lg,\La)$ is the root system of $\Lg$ with respect to $\La$. Note that $\check\alpha(-1)\in I_N$. Then $\cH_{W_{\chi_m}^0}$ is the Hecke algebra associated to the Coexter group $W_{\chi_m}^0$ with parameter $-1$. Let $B_{\chi_m}^0\subset B_N$ be the inverse image of $W_{\chi_m}^0\subset W$ under the natural map $B_N\to W=S_N$. Then  $\widetilde{B}_{\chi_m}^0=I_N\rtimes B_{\chi_m}^0$. In our setting, one readily checks that $W_{\chi_m}^0=\langle s_i,i\neq m\rangle$, $\cH_{W_{\chi_m}^0}=\cH_{\chi_m,-1}$ and $B_{\chi_m}^0=B_{m,N-m}$ (here we use the notations in~\S\ref{ssec-representations of tB}). The action of $\widetilde{B}_{\chi_m}^0$ on $(\bC_{\chi_m}\otimes\cH_{W_{\chi_m}^0})$ is given by $I_N$ acting via the character $\chi_m$ and ${B}_{\chi_m}^0$ acting via the quotient map $\bC[{B}_{\chi_m}^0]\to\cH_{W_{\chi_m}^0}$. Thus we conclude
\beq\label{FT of nearby cycles}
\mathfrak{F}(\cP_{\chi_m})\cong\on{IC}(\Lg_1,\cL_{\chi_m})\,,
\eeq
where $\cL_{\chi_m}$ is the $K$-equivariant local system on $\Lg_1^{rs}$ corresponding to the representations $L_{\chi_m}$ of $\pi_1^K(\Lg_1^{rs})=\widetilde{B}_{N}$ defined in~\eqref{induced module L}.

By Lemma~\ref{composition factors} the IC sheaves $\IC(\Lg_1,\cT({\emptyset;\mu^1,\mu^2}))$ and $\IC\left(\Lg_1,\cT\left({\emptyset;\mu,\mu}\right)^\omega\right)$ are composition factors of the $\on{IC}(\Lg_1,\cL_{\chi_m})$. Hence~\eqref{FT full support} follows from~\eqref{FT of nearby cycles}.

Now let $(\nu;\mu^1,\mu^2)\in\Sigma_{N}$ with $|\nu|=m>0$. Let
\beqn
\cK(\rho_\nu\boxtimes V_{\mu^1,\mu^2})
\eeqn
 denote the irreducible $L_K$-equivariant local system on $\Ll_1^{rs}$ associated to the irreducible representation  of $\pi_1^{L_K}(\Ll_1^{rs})$ obtained as a pullback 
 of $\rho_\nu\boxtimes V_{\mu^1,\mu^2}$ via the map $\pi_1^{L_K}(\Ll_1^{rs})\cong B_m\times\widetilde{B}_{N-2m}\twoheadrightarrow S_m\times\widetilde{B}_{N-2m}$.

 By Proposition~\ref{prop-induction2}, we have that 
\beq\label{induced IC sheaf}
\IC\left(\Lg_1^m,\cT({\nu;\mu^1,\mu^2})\right)=\on{Ind}_{\Ll^m_1\subset\Lp^m_1}^{\Lg_1}\IC\left(\Ll_1,\cK(\rho_\nu\boxtimes V_{\mu^1,\mu^2})\right).
\eeq 

Since Fourier transform commutes with induction  (see~\eqref{commutativity}), it suffices to show that the Fourier transform of 
$\IC\left(\Ll_1,\cK(\rho_\nu\boxtimes V_{\mu^1,\mu^2})\right)$ is supported on an $L_K$-nilpotent orbit in $\Ll_1$. This follows from the classical Springer correspondence for $\mathfrak{gl}(m)$ and~\eqref{FT full support} applied to the symmetric pair $(\SL(N-2m),\SO(N-2m))$ (see~\eqref{LK and l1}). 

The proof for $\IC(\Lg_1^m,\cT(\nu;\mu,\mu)^\omega)$, $\IC(\Lg_1^{n,\omega},\cT(\nu;\emptyset,\emptyset)^{\omega})$ proceeds in the same manner; in the latter case one uses the corresponding $\theta$-stable Levi and parabolic subgroups. We omit the details. 
\end{proof}

\subsection{More on induction}Let $(\nu;\mu^1,\mu^2)\in\Sigma_{N}$. Assume that $|\nu|=m>0$. Let $L^m\subset P^m$  be as in \S\ref{sec-parabolic}. Recall that $L^m_K\cong \GL(m)\times \SO(N-2m)$ and $\Ll^m_1\cong\mathfrak{gl}(m)\oplus \mathfrak{sl}(N-2m)_1$.

A nilpotent $L^m_K$-orbit in $\Ll^m_1$ is given by a nilpotent orbit in $\mathfrak{gl}(m)$ and a nilpotent $\SO(N-2m)$-orbit in $\mathfrak{sl}(N-2m)_1$. Thus the nilpotent $L^m_K$-orbits in $\Ll^m_1$ are parametrized by $\cP(m)\times\cP(N-2m)$, with extra labels $\I$ and $\II$ for partitions in $\cP(N-2m)$ with all parts even. For $\alpha\in\cP(m)$ and $\beta\in\cP(N-2m)$, we denote by $\cO_{\alpha,\beta}$ (or $\cO_{\alpha,\beta}^\omega$) the nilpotent $L_K^m$-orbit  in $\Ll^m_1$ given by the nilpotent orbit $\cO_\alpha$ in $\mathfrak{gl}(m)$ and the nilpotent $\SO(N-2m)$-orbit $\cO_{\beta}$ (or $\cO_\beta^\omega$) in $\mathfrak{sl}(N-2m)_1$. 

In the following we will omit the labels $\I$ and $\II$ with the understanding that everything should have corresponding labels, for example, $\cO_\lambda^\omega=\on{Ind}_{\Ll_1^m\subset\Lp_1^m}^{\Lg_1}\cO_{\alpha,\beta}^\omega$ etc.

\begin{proposition}\label{prop-induction}
 Let $\alpha\in\cP(m)$ and $\beta\in\cP(N-2m)$. Let $\cO_\lambda=\on{Ind}_{\Ll_1^m\subset\Lp_1^m}^{\Lg_1}\cO_{\alpha,\beta}$, i.e., $\lambda_i=\beta_i+2\alpha_i$. Assume that $u\in\cO_{\alpha,\beta}$ and $v\in\cO_\lambda\cap(u+(\Ln_{P^m})_1)$. We have a natural surjective map
 \beqn
\psi:A_K(v)\twoheadrightarrow A_{L^m_K}(u).
\eeqn
Moreover, let $\bC\boxtimes\cE$ be an $L^m_K$-equivariant irreducible local system on $\cO_{\alpha,\beta}$ and let $\tilde{\cE}$ be the $K$-equivariant local system on $\cO_\lambda$ obtained from $\bC\boxtimes\cE$ via the map $\psi$ above. Then $\IC(\cO_\lambda,\tilde{\cE})$ is a direct summand of $\on{Ind}_{\Ll^m_1\subset\Lp^m_1}^{\Lg_1}\on{IC}(\cO_{\nu,\mu},\bC\boxtimes\cE)$.

\end{proposition}
\begin{corollary}\label{coro-ind}
If moreover $(\cO_\mu,\cE)\in\cA_{N-2m}$ is a pair such that $\fF(\IC(\cO_\mu,\cE))$ has full support, then 
we have
\beqn
\on{Ind}_{\Ll^m_1\subset\Lp^m_1}^{\Lg_1}\on{IC}(\cO_{\nu,\mu},\bC\boxtimes\cE)\cong\IC(\cO_\lambda,\tilde{\cE}).
\eeqn
\end{corollary}

As before let us now write $L=L^m$ and $P=P^m$ etc. We begin the proof of the above proposition with the following lemma.
\begin{lemma}\label{lemma-induced orbit}

The map
\beqn
\gamma:K\times ^{P_K}(\bar\cO_{\alpha,\beta}+(\Ln_P)_1)\to\bar\cO_\lambda
\eeqn
is generically one-to-one.
\end{lemma}
\begin{proof}
Let $x_0\in\cO_\lambda$. We can and will assume that $x_0\in\cO_{\alpha,\beta}+(\Ln_P)_1$. We show that $\gamma^{-1}(x_0)$ is a point. Assume that $\gamma(k,x)=x_0$. i.e. $\on{Ad}k(x)=x_0$. Then $x\in\cO_{\alpha,\beta}+(\Ln_P)_1$. Let $\widetilde\cO_\lambda$ (resp. $\widetilde\cO_{\alpha,\beta}$) be the (unique) $G$-orbit (resp. $L$-orbit) in $\Lg$ (resp. $\Ll$) that contains $\cO_\lambda$ (resp. $\cO_{\alpha,\beta}$). We have that 
\beqn
\widetilde\cO_\lambda=\on{Ind}_{\Ll}^{\Lg}\widetilde\cO_{\alpha,\beta}
\eeqn
in the notation of Lusztig and Spaltenstein \cite{LS}. By \cite[Theorem 1.3]{LS}, we have $Z_G^0(x_0)\subset P$. In fact, we have that $Z_G(x_0)\subset P$. This can be seen by enlarging the group $G$ to $GL(N)$ and using the fact that $Z_{GL(N)}(x_0)$ is connected. Thus $Z_K(x_0)\subset P_K$. Furthermore, $\widetilde\cO_\lambda\cap(\widetilde\cO_{\alpha,\beta}+\Ln_P)$ is a single orbit under $P$. Thus there exists $p\in P$ such that $\on{Ad}p(x)=x_0$. It follows that $k^{-1}p\in Z_G(x_0)\subset P$. Thus $k\in P\cap K=P_K$. Now we have that $(k,x)=(1,\on{Ad}k(x))=(1,x_0)$.
\end{proof}

\begin{proof}[Proof of Proposition~\ref{prop-induction}]
Note that the proof of the above lemma shows that $Z_G(v)=Z_P(v)$. We have $Z_P(v)\subset Z_L(u)U_P$.
Thus $Z_K(v)=Z_{P_K}(v)\subset Z_{L_K}(u)(U_P\cap K)$. It follows that we have a natural projection map
\beqn
Z_{K}(v)/Z_{K}^0(v)=Z_{P_K}(v)/Z_{P_K}^0(v)\to Z_{L_K}(u)/Z_{L_K}^0(u).
\eeqn
We show that  this gives us the desired map $\psi$.  Following \cite{LS}, we have that $Z_{L_K}(u)(U_P\cap K)$ has a dense orbit, i.e. the orbit of $v$, in the irreducible variety $u+(\Ln_P)_1$. Thus $Z_{P_K}(v)=Z_K(v)$ meets all the irreducible components of $Z_{L_K}(u)(U_P\cap K)$, which implies that $\psi$ is surjective.

It is easy to see that 
\beq\label{support of induced}
\on{supp}\left(\on{Ind}_{\Ll_1\subset\Lp_1}^{\Lg_1}\on{IC}(\cO_{\alpha,\beta},\bC\boxtimes\cE)\right)=\bar\cO_\lambda.
\eeq
The proposition follows from the definition of parabolic induction functor and Lemma~\ref{lemma-induced orbit}.
\end{proof}

\begin{remark}
The proof of Lemma~\ref{lemma-induced orbit} and the existence and surjectivity of the map $\psi$ in Proposition~\ref{prop-induction} works for any $\theta$-stable Levi contained in a $\theta$-stable parabolic subgroup.
\end{remark}

\begin{proof}[Proof of Corollary~\ref{coro-ind}]
Note that the assumption implies that  $\fF\left(\on{IC}(\cO_{\alpha,\beta},\bC\boxtimes\cE)\right)$ has full support, i.e. $\on{IC}(\cO_{\alpha,\beta},\bC\boxtimes\cE)=\IC(\Ll_1,\cG)$ for some irreducible $L_K$-equivariant local system $\cG$ on $\Ll_1^{rs}$. We have that
\beqn
\fF\left(\on{Ind}_{\Ll_1\subset\Lp_1}^{\Lg_1}\on{IC}(\cO_{\alpha,\beta},\bC\boxtimes\cE)\right)=\on{Ind}_{\Ll_1\subset\Lp_1}^{\Lg_1}\fF\left(\on{IC}(\cO_{\alpha,\beta},\bC\boxtimes\cE)\right)=\on{Ind}_{\Ll_1\subset\Lp_1}^{\Lg_1}\IC(\Ll_1,\cG).
\eeqn
It suffices to show that $\on{Ind}_{\Ll_1\subset\Lp_1}^{\Lg_1}\IC(\Ll_1,\cG)$ is irreducible. This follows from the definition of the induction functor and Proposition \ref{small}.
\end{proof}

\begin{corollary}\label{coro-full supp1}
The Fourier transform of a nilpotent orbital complex $\IC(\cO,\cE)\in\cA_N$ has full support, i.e., $\on{supp}\fF(\IC(\cO,\cE))=\Lg_1$, if and only if it is not of the form $\on{Ind}_{\Ll_1\subset\Lp_1}^{\Lg_1}\IC(\cO',\cE')$ where $\on{supp}\fF(\IC(\cO',\cE'))=\Ll_1$, and $L\subset P$ is a pair chosen as in \S\ref{sec-parabolic}. 
\end{corollary}
\begin{proof}
The only if part follows from the facts that Fourier transform commutes with parabolic induction and that $\on{supp}\on{Ind}_{\Ll_1\subset\Lp_1}^{\Lg_1}A\subsetneq \Lg_1$. The if part follows from~\eqref{induced IC sheaf},~\eqref{coro-ind} and Theorem~\ref{main theorem}.
\end{proof}

\begin{corollary}\label{coro-full supp2}

Let $\lambda=(\lambda_1\geq\lambda_2\geq\cdots)\in\cP(N)$. 

(1) If $\lambda_i-\lambda_{i+1}\geq 3$ for some $i$, then $\on{supp}\fF(\IC(\cO_\lambda,\cE))\neq\Lg_1$ for any $K$-equivariant local system $\cE$ on $\cO_\lambda$. The same holds for $\cO_\lambda^\omega$ if $\lambda$ has only even parts.

(2) Suppose that $\lambda_i-\lambda_{i+1}\leq 2$ for all $i$. Let $f_\lambda$ be the number of different sizes of parts of $\lambda$, and $g_\lambda$ the number of $i$'s such that $\lambda_{i}-\lambda_{i+1}=2$.
\begin{enumerate}
\item[(a)] If at least one part of $\lambda$ is odd, then    
 there are $2^{f_\lambda-1-g_\lambda}$ irreducible $K$-equivariant local systems $\cE$ on $\cO_\lambda$ such that $\on{supp}\fF(\IC(\cO_\lambda,\cE))=\Lg_1$.
 \item[(b)] If all parts of $\lambda$ are even, then there is exactly one irreducible $K$-equivariant local system $\cE^\omega$ on each orbit $\cO_\lambda^{\omega}$, $\omega=\I,\II$,  such that $\on{supp}\fF(\IC(\cO_\lambda^\omega,\cE^\omega))=\Lg_1$.
 \end{enumerate}
 In particular, if  $\lambda_i-\lambda_{i+1}\leq 1$ for all $i$, then $\on{supp}\fF(\IC(\cO_\lambda,\cE))=\Lg_1$ for any $K$-equivariant local system $\cE$ on $\cO_\lambda$. 

\end{corollary}
\begin{proof}

(1) Assume that $\lambda_{i_0}-\lambda_{i_{0}+1}\geq 3$. Let $m=i_0$, $\alpha=1^{i_0}$, $\beta=(\lambda_1-2,\ldots,\lambda_{i_0}-2,\lambda_{i_0+1},\ldots)$. Then $\cO_\lambda=\on{Ind}_{\Ll^m_1\subset\Lp_1^m}^{\Lg_1}\cO_{\alpha,\beta}$.
Let $u\in\cO_{\alpha,\beta}$ and $v\in\cO_\lambda\cap(u+(\Ln_{P^m})_1)$. Note that $A_{K}(v)\cong A_{L^m_K}(u)$. 
It then follows from Proposition~\ref{prop-induction} that for each irreducible $K$-equivariant local system $\cE$ on $\cO_\lambda$, $\IC(\cO_\lambda,\cE)$ is a direct summand of $\on{Ind}_{\Ll^m_1\subset\Lp_1^m}^{\Lg_1}\IC(\cO_{\alpha,\beta},\cE_0)$ for some irreducible $L_K$-equivariant local system $\cE_0$ on $\cO_{\alpha,\beta}$. As before, this shows that $\fF(\IC(\cO_\lambda,\cE))$ has smaller support.

In the case when $\lambda$ has only even parts, we let $\cO_\lambda^\omega=\on{Ind}_{\Ll^m_1\subset\Lp_1^m}^{\Lg_1}\cO_{\alpha,\beta}^\omega$, if $m<N/2$, and we let $\cO_\lambda^\omega=\on{Ind}_{\Ll^{n,\omega}_1\subset\Lp^{n,\omega}_1}^{\Lg_1}\cO_{\alpha,\beta}$, if $m=N/2=n$, where $\omega=\I,\II$. The proof for $\cO_\lambda^\omega$ then proceeds in the same way.

(2) We argue by induction on $g_\lambda$. If $g_\lambda=0$, then (2) follows from~\eqref{support of induced} and Corollary~\ref{coro-full supp1}. Assume by induction hypothesis  that (2) holds for all $\mu$ with $g_{\mu}<g_\lambda$. 

Assume first that $\lambda$ has at least one odd part. Suppose that $i_1,\ldots,i_k$ are such that $\lambda_{i_j}-\lambda_{i_{j}+1}=2$, where $k=g_\lambda$.  

Let $a=(a_1\geq a_2\geq \cdots\geq a_k\geq 0)$ be a partition such that $a\neq\emptyset$, $a_k\leq 1$, and $a_l\leq a_{l+1}-1$. Note that the number of such partitions is $2^k-1$. Consider a partition $\mu(a)$ such that $\mu_l=\lambda_l-2a_j$ for $l\in[i_{j-1}+1,i_j]$. Then $\mu(a)$ satisfies that $\mu(a)_i-\mu(a)_{i+1}\leq 2$ and $g_{\mu(a)}<g_\lambda$. Moreover, $\mu$ has at least one odd part, and $f_{\lambda}-g_\lambda=f_{\mu(a)}-g_{\mu(a)}$. Let $m=\sum_{j=1}^k{i_j}$. We have that
\beqn
\on{Ind}_{\Ll^m_1\subset\Lp^m_1}^{\Lg_1}\cO_{a,\mu(a)}=\cO_\lambda.
\eeqn
By induction hypothesis, there are $2^{f_\lambda-g_\lambda-1}$ irreducible $K$-equivariant local systems $\cE$ on $\cO_{a,\mu(a)}$ such that $\fF(\on{IC}(\cO_{a,\mu(a)},\cE)$ has full support. By Corollary~\ref{coro-ind}, we have that 
 \beqn
\on{Ind}_{\Ll^m_1\subset\Lp^m_1}^{\Lg_1}\on{\IC}(\cO_{a,\mu(a)},\cE)=\on{\IC}(\cO_\lambda,\tilde\cE).
\eeqn
This gives rise to $(2^{k}-1)\cdot2^{f_\lambda-g_\lambda-1}=2^{f_\lambda-1}-2^{f_\lambda-g_\lambda-1}$ irreducible $K$-equivariant local systems $\tilde\cE$ on $\cO_{\lambda}$ such that $\fF(\on{IC}(\cO_{\lambda},\tilde\cE)$ has smaller support (with $a$ varying). 

The case when all parts of $\lambda$ even can be argued in the same way. Note that in this case $g_\lambda=f_\lambda$.

Let us write $m_\lambda$ (resp. $m_\lambda^\omega$, $\omega=\I,\II$) for the number of irreducible $K$-equivariant local systems $\tilde\cE$ on $\cO_{\lambda}$ (resp. $\cO_{\lambda}^\omega$) such that $\fF(\on{IC}(\cO_{\lambda},\tilde\cE)$ (resp. $\fF(\on{IC}(\cO_{\lambda}^\omega,\tilde\cE)$) has full support when at least one part of $\lambda$ is odd (resp. when all parts of $\lambda$ are even).

We conclude from the discussion above that
\beq\label{estimate}
\begin{gathered}
m_\lambda\leq 2^{f_\lambda-g_\lambda-1}\text{ if }\lambda\text{ has at least one odd part},\\\text{ resp. }
m_\lambda^\omega\leq 1\text{ if all parts of }\lambda\text{ are even}.
\end{gathered}
\eeq
Theorem~\ref{main theorem} implies that the number of pairs $\on{IC}(\cO,\cE)\in\cA_N$ such that $\on{supp}\fF(\on{IC}(\cO,\cE))=\Lg_1$ is $d(n)$ (see~\eqref{def of d(k)}), when $N=2n+1$, and $e(n)$ (see~\eqref{def of e(k)}), when $N=2n$. 
In view of~\eqref{estimate} and claim (1) of the corollary, it suffices to show that
\beq\label{est sum}
\sum_{\substack{\lambda\in\cP(2n+1)\\\lambda_i-\lambda_{i+1}\leq 2}}2^{f_\lambda-g_\lambda-1}=d(n)
,\ \ 
\sum_{\substack{\lambda\in\cP(2n),\lambda_i-\lambda_{i+1}\leq 2,\\\text{ not all parts of $\lambda$ even}}}2^{f_\lambda-g_\lambda-1}+2q(n)=e(n).
\eeq
This can be seen as follows.  Note that when $N$ is even, the number of  orbits of the form $\cO_\lambda^\omega$, where all parts of $\lambda$ are even and $\lambda_i-\lambda_{i+1}\leq 2$,  is $2q(n)$.
We know that 
\beqn
\text{$d(n)$ = Coefficient of $x^{2n+1}$ in }\frac{1}{2}\prod_{s\geq 1}(1+x^s)^2.
\eeqn
\beqn
e(n)=\frac{3}{2}q(n)+\text{Coefficient of $x^{2n}$ in } \frac{1}{2}\prod_{s\geq 1}(1+x^s)^2.
\eeqn
A partition $\lambda$ satisfies that $\lambda_i-\lambda_{i+1}\leq 2$ if and only if each part of the transpose partition $\lambda'$ has multiplicity at most  $2$. We have $f_\lambda=f_{\lambda'}$ and $g_\lambda$ equals the number of parts in $\lambda'$ with multiplicity $2$. 
It is easy to see that each $\lambda'$ whose parts have multiplicity at most $2$ appears in $\prod_{s\geq 1}(1+x^s)^2$ exactly $2^{f_\lambda-g_\lambda}$ times. Hence~\eqref{est sum} follows.
\end{proof}

\begin{remark}
In \cite[Conjecture 1.2]{CVX1}, we conjectured that one can obtain 
all nilpotent orbital complexes by induction
from those of smaller groups whose Fourier transforms have full support. 
This conjecture follows from Corollary \ref{coro-full supp1}.

\end{remark}

\section{Cohomology of Hessenberg varieties}\label{sec-Hess}

Hessenberg varieties, defined generally in \cite{GKM}, arise naturally in our setting (for details, see \cite{CVX2}). In particular, they arise as fibers of maps $\pi$ and $\check{\pi}$ in the following diagram
\beqn
\xymatrix{&K/P_K\times\Lg_1&\\K\times^{P_K}E\ar[d]^-{\pi}\ar[ur]&&K\times^{P_K}E^\perp\ar[d]_-{\check{\pi}}\ar[ul]\\\cN_1&&\Lg_1}
\eeqn
where $P_K=P\cap K$ for a $\theta$-stable parabolic subgroup $P$ of $G$, $E$ is a $P_K$-stable subspace of $\Lg_1$ consisting of nilpotent elements, and $E^\p$ is the orthogonal complement of $E$ in $\Lg_1$ via a $K$-invariant non-degenerate form on $\Lg_1$. The generic fibers of maps $\check{\pi}$ are Hessenberg varieties. 

In this section we discuss an application of our result to 
cohomology of Hessenberg varieties. 
Let us fix $s\in\fg_1^{rs}$ and consider the corresponding Hessenberg 
variety \[\on{Hess}:=\check\pi^{-1}(s)=\{gP_K\in K/P_K\,|\,g^{-1}sg\in E^\perp\}.\] 
The centralizer $Z_K(s)$ acts naturally on $\on{Hess}$ and it induces 
an action of the component group $\pi_0(Z_K(s))\cong I_N$ on the 
cohomology groups $H^*(\on{Hess},\bC)$. Let 
\[\oh^*(\on{Hess},\bC)=\bigoplus_{\chi\in I_N^\vee} \oh^*(\on{Hess},\bC)_\chi\]
be the eigenspace decomposition with respect to the action 
of $I_N$.

\begin{definition}
The stable part $\oh^*(\on{Hess},\bC)_{\on{st}}$ of $\oh^*(\on{Hess},\bC)$ is the 
direct summand\linebreak $\oh^*(\on{Hess},\bC)_{\chi_{\on{triv}}}$ where $\chi_{\on{triv}}\in I_N^\vee$
is the trivial character. 
\end{definition}

For simplicity we now assume $\check\pi$ is onto. 
In this case $\check\pi$ is 
smooth over $\fg_1^{rs}$ (e.g. see \cite[Lemma 2.1]{CVX2}) and  
the equivariant fundamental group $\pi_1^K(\fg_1^{rs},s)\cong I_N\rtimes B_N$
acts on $H^*(\on{Hess},\bC)$ by the monodromy action. 
Recall that for $\chi\in I_N^\vee$, $B_\chi$
stands for the stabilizer of $\chi$ in $B_N$.
Clearly, each summand $H^*(\on{Hess},\bC)_\chi$ is 
stable under the action of 
$B_{\chi}$.
Let $\chi_m\in I_N^\vee$, $B_{\chi_m}$, and 
$B_{m,N-m}$ be as in \S\ref{ssec-representations of tB}.
Assume that $\chi$ is in the $B_N$-orbit of $\chi_m$. 
Then for any $b\in B_N$ with $b.\chi=\chi_m$ we have an 
isomorphism $\iota_b:B_\chi\cong B_{\chi_m}, u\to bub^{-1}$.
Note that
$\chi_{\on{triv}}=\chi_0$ and 
$B_{\chi_m}=B_{m,N-m}$ except 
when $N$ is even and $m=N/2$. 
In that case,  $B_{m,N-m}$ is an index two subgroup of $B_{\chi_m}$.

Recall the 
algebra $\cH_{\chi_m,-1}=\cH_{m,-1}\otimes \cH_{N-m,-1}$
and their representations $D_{\mu^1}\otimes D_{\mu^2}$ introduced in \S\ref{ssec-representations of tB}. 
Each $\cH_{\chi_m,-1}$ is a quotient of the group algebra $\bC[B_{m,N-m}]$
and  $\cH_{\chi_{0},-1}=\cH_{\chi_{\on{triv}},-1}=\cH_{N,-1}$
is the Hecke algebra of $S_N$ at $q=-1$.

\begin{theorem}\label{coh of Hess}
\begin{enumerate}
\item 
Let $\chi_m\in I_N^\vee$ be the representatives
of $B_N$-orbtis in \S\ref{ssec-representations of tB}.
To every $\chi\in I_N^\vee$ in the orbit of $\chi_m$ and 
an element $b\in B_N$ satisfying $b(\chi)=\chi_m$, 
the monodromy action of $b$ on $\oh^*(\on{Hess},\bC)$ induces 
an isomorphism 
$\oh^{*}(\on{Hess},\bC)_\chi\cong \oh^{*}(\on{Hess},\bC)_{\chi_m}$
compatible with the actions of $B_\chi\stackrel{\iota_b}\cong B_{\chi_m}$ on 
both sides.\\
\item 
The action of $\bC[B_{m,N-m}]$ on
$\oh^*(\on{Hess},\bC)_{\chi_m}$ factors through the algebra $\cH_{\chi_{m},-1}$ 
and the resulting representation is a direct sum of 
$D_{\mu^1}\otimes D_{\mu^2}$, $\mu^1\in\cP_2(m)$, 
$\mu^2\in\cP_2(N-m)$.
In particular, the stable part $\oh^*(\on{Hess},\bC)_{\on{st}}$
is generated by irreducible representations of 
the Hecke algebra of $S_N$ at $q=-1$.
\end{enumerate}

\end{theorem}
\begin{proof}
Part (1) is clear. 
To prove part (2) we proceed as follows. 
By the decomposition theorem
$\pi_*\bC$ is a direct sum of shifts 
of nilpotent orbital complexes. Since 
$\fF(\pi_*\bC)\cong\check{\pi}_*\bC$ (up to shift), 
Theorem~\ref{main theorem} implies that
a generic stalk of 
$\check{\pi}_*\bC$,
which is isomorphic to $H^*(\on{Hess},\bC)$,
is a direct sum of 
the local systems 
 $V_{\mu^1,\mu^2}=\on{Ind}_{\bC[B_{m,N-m}]}^{\bC[B_N]}D_{\mu_1}\otimes D_{\mu_2}$
introduced in (\ref{induced module L}). 
Since $I_N$ acts on $V_{\mu^1,\mu^2}$ by the formula 
$a.(b\otimes v)=\left((b.\chi_m)(a)\right)( b\ot v)$ for $a\in I_{N}$, $b\in B_{N}$ and $v\in D_{\mu_1}\otimes D_{\mu_2}$, we have 
$(V_{\mu^1,\mu^2})_\chi\cong
D_{\mu_1}\otimes D_{\mu_2}$. The theorem follows.
\end{proof}

\begin{example}
Let $C$ be the hyper-elliptic curve with affine equation $y^2=\prod_{j=1}^N(x-a_j)$ (here $a_i\neq a_j$ for $i\neq j$). 
Assume 
$N=2n+2$ is even. Then according to \cite[Section 2.3]{CVX3} the Jacobian $\on{Jac}(C)$ is 
an example of Hessenberg variety and 
the monodromy action of 
$\pi_1(\fg_1^{rs},s)$ factors through $B_N$, that is,
$\oh^*(\on{Jac}(C),\bC)=\oh^*(\on{Jac}(C),\bC)_{\on{st}}$.
Let $\mu_k=(N-k,k)\in\cP_2(N)$ and $D_{\mu_k}$ be the corresponding 
representation of $\cH_{N,-1}$. 
Using \cite{A}, one can check that the induced action of the group algebra 
$\bC[B_N]$ on 
$\oh^i(\on{Jac}(C),\bC)$
factors through $\cH_{N,-1}$ and for $i\leq n$ the resulting representation of $\cH_{N,-1}$
is isomorphic to 
\beqn
\oh^i(\on{Jac}(C),\bC)\cong\bigoplus_{j=0}^{[i/2]} D_{\mu_{i-2j}}
\eeqn
with the primitive part
$
\oh^i(\on{Jac}(C),\bC)_{\mathrm{prim}}\cong D_{\mu_i}$.

\end{example}

\begin{remark}
It would be nice to have an explicit  
decomposition of $H^*(\on{Hess},\bC)_{\chi_m}$ into irreducible representations of $\cH_{\chi_m,-1}$. 
For this one  needs finer information for the bijection in Theorem \ref{main theorem}
(see Section \ref{sec-conj}).
In \cite{CVX1,CVX3}, we establish an explicit bijection
for certain nilpotent orbital complexes and we 
work out an explicit decomposition for the cohomology of the Hessenberg varieties that are isomorphic to Fano varieties of $k$-planes in smooth complete intersections of 
two quadrics in projective space. 
\end{remark}

\section{Representations of \texorpdfstring{$\cH_{N,-1}$}{lg}}\label{Rep}
In this section we show that \emph{all} irreducible representations 
of the Hecke algebra $\cH_{N,-1}$ come from geometry. Indeed they all appear in intersection cohomology of a Hessenberg variety with coefficient in a local system. In particular, this shows that all irreducible representations of $\cH_{N,-1}$ carry a Hodge structure. In particular, the irreducible representations of $\cH_{N,-1}$ can be viewed as variations of Hodge structure.

Let $\cO$ be a nilpotent $K$-orbit on $\fg_1$ and $\cL$ an irreducible 
$K$-equivariant local system on $\cO$.
We call $(\cO,\cL)$ a nilpotent pair.
Following \cite{LY}, we associate to each nilpotent pair $(\cO,\cL)$ two families of 
Hessenberg varieties $\on{Hess}_{\cL,\pm 1}\to\fg_1$ together with 
local systems $\hat\cL_{\pm 1}$ on open subsets $\accentset{\circ}{\on{Hess}}_{\cL,\pm1}\subset
\on{Hess}_{\cL,\pm 1}$.

Let $x\in\fg_1$ be a nilpotent element in $\cO$.
Choose a normal $\on{sl}_2$-triple $\{x,h,y\}$ 
and let 
\beqn
\text{$\fg(i)=\{v\in\fg|[h,v]=iv\}$, $\fg_0(i)=\fg(i)\cap\fg_0$, and $\fg_1(i)=\fg(i)\cap\fg_1$.}
\eeqn
For any $N\in\bZ$ we write $\underline N\in\{0,1\}$ for its image in $\bZ/2\bZ$.
Define 
\beqn
\text{$\fp_N^x=\bigoplus_{k\geq 2N}\fg_{\underline N}(k)$, $\fl_N^x=\fg_{\underline N}(2N)$,
and $\fl^x=\bigoplus_{N\in\bZ}\fl_N^x$.}
\eeqn
One can check that 
$\fl^x\subset\fg$ is a graded Lie subalgebra of $\fg$ and $x\in\fl_1^x=\fg_1(2)$.
Let $L_0^x\subset K$ be the reductive subgroup with Lie algebra 
$\fl_0^x=\fg_0(0)$. By \cite[2.9(c)]{LY}, the restriction 
\beqn
\cL'_1:=\cL|_{\fl_{1}^x}
\eeqn
is an irreducible $L_0^x$-equivariant local system 
on the unique open $L_0^x$-orbit
$\accentset{\circ}{\fl^{x}_{1}}$
on $\fl^x_{1}$.

According to \cite{L1}, there exists a graded parabolic subalgebra 
$\mathfrak q=\bigoplus_{N\in\bZ}\mathfrak q_N$ of $\fl^x$, a Levi subalgebra $\fm
=\bigoplus_{N\in\bZ}\fm_N$ of $\mathfrak q$, and a 
cuspidal local system $\cL_1$ on the open $M_0$-orbit
$\accentset{\circ}{\mathfrak m}_1$ of $\fm_1$ (here $M_0$ is the reductive subgroup of $L_0^x$
with Lie algebra $\fm_0$) such that 
\beqn
\text{some shift of 
the IC-complex $\on{IC}(\fl_1^x,\cL_1')$ is a direct summand of 
$\on{Ind}_{\fm_1\subset\mathfrak q_1}^{\fl_1^x}\on{IC}(\fm_1,\cL_1)$.} 
\eeqn
In addition, we have  
\beqn
\mathfrak F(\on{IC}(\fm_1,\cL_1))\cong\on{IC}(\fm_{-1},\cL_{-1})
\eeqn
where $\cL_{-1}$ is a cuspical local system on the unique open orbit $\accentset{\circ}{\fm}_{-1}\subset\fm_{-1}$.

Define $\hat{\mathfrak q}_N$ to be the pre-image of 
$\mathfrak q_N$ under the projection map
$\fp_N^x\to\fl_N^x$. 
Let $Q_K\subset K$ be the parabolic subgroup with Lie algebra 
$\hat{\mathfrak q}_0$. Denote by $\accentset{\circ}{\hat{\mathfrak q}}_{\pm 1}$ the preimage of 
$\accentset{\circ}{\fm}_{\pm1}$ under the projection map 
$\hat{\mathfrak q}_{\pm 1}\to\mathfrak q_{\pm 1}\to \mathfrak m_{\pm 1}$.
The group $Q_K$ acts naturally on
$\hat{\mathfrak q}_{\pm 1}$ and $\accentset{\circ}{\hat{\mathfrak q}}_{\pm 1}$
and 
we define 
\[\on{Hess}_{\cL,\pm 1}:=K\times^{Q_K}\hat{\mathfrak q}_{\pm 1},\\\ 
\accentset{\circ}{\on{Hess}}_{\cL,\pm 1}:=K\times^{Q_K}\accentset{\circ}{\hat{\mathfrak q}}_{\pm 1}.\]
Let 
\beqn
\pi_{\cL,\pm1}:\on{Hess}_{\cL,\pm 1}\to\fg_1, (x,v)\to xvx^{-1}
\eeqn and 
let $\accentset{\circ}{\pi}_{\cL,\pm1}$ be its restriction to $\accentset{\circ}{\on{Hess}}_{\cL,\pm 1}$. 
For any $s\in\fg_1$, we denote by 
$\on{Hess}_{\cL,\pm1,s}$ and $\accentset{\circ}{\on{Hess}}_{\cL,\pm1,s}$ the fiber of 
$\pi_{\cL,\pm1}$ and $\accentset{\circ}{\pi}_{\cL,\pm1}$ over $s$, respectively.

There are natural maps 
\[h_{\cL,\pm 1}:\on{Hess}_{\cL,\pm 1}\to[\fm_{\pm1}/M_0], \ \ \ 
\accentset{\circ}{h}_{\cL,\pm 1}:\accentset{\circ}{\on{Hess}}_{\cL,\pm 1}\to[\accentset{\circ}{\fm}_{\pm1}/M_0]\]
sending $(k,v)\in\on{Hess}_{\cL,\pm 1}=K\times^{Q_K}\hat{\mathfrak q}_{\pm 1}$ to 
$\bar v$, the image of $v\in\hat{\mathfrak q}_{\pm 1}$ under the map
$\hat{\mathfrak q}_{\pm 1}\to\fm_{\pm1}\to
[\fm_{\pm1}/M_0]$.
We define the following local system 
\[\hat\cL_{\pm1}:=(\accentset{\circ}{h}_{\cL,\pm 1})^*\cL_{\pm 1}\]
on $\accentset{\circ}{\on{Hess}}_{\cL,\pm 1}$. 
Here we view the $M_0$-local systems $\cL_{\pm1}$ as sheaves on $[\accentset{\circ}{\fm}_{\pm1}/M_0]$.

\begin{example}
Consider the nilpotent pair $(\cO,\cL=\cL_{\on{triv}})$ where 
$\cL_{\on{triv}}$ is the trivial local system on $\cO$. 
Using \cite[Proposition 7.3]{L1} one can check that in this case 
$\mathfrak q=\oplus_{N\in\bZ}\mathfrak q_N$ is a Borel 
subalgebra of $\mathfrak l^x$
and $\mathfrak m=\oplus_{N\in\bZ}\mathfrak m_N$ is a Cartan subalgebra.
Moreover the grading on $\mathfrak m$ is concentrated in 
degree zero, i.e., $\mathfrak m=\mathfrak m_0$, and the 
cuspidal local system $\cL_{\pm 1}$ is the skyscraper sheaf supported 
on $\mathfrak m_{\pm 1}=\{0\}$. It follows that  in this case 
${\on{Hess}}_{\cL_{\on{triv}},\pm 1}=\accentset{\circ}{\on{Hess}}_{\cL_{\on{triv}},\pm 1}$
and $\hat\cL_{\pm1}$ is the constant local system.
\end{example}

In \cite[\S 7]{LY}, the authors prove the following:
\beq\label{LY2}
\begin{gathered}
(\pi_{\cL,-1})_*\on{IC}(\on{Hess}_{\cL,-1},\hat\cL_{-1})\ \ 
\text{is the Fourier transform of}\ \
(\pi_{\cL,1})_*\on{IC}(\on{Hess}_{\cL,1},\hat\cL_1).
 \end{gathered}
\eeq
\beq\label{LY}
\begin{gathered}
\text{Some shift of}\ \on{IC}(\cO,\cL)\ (\text{resp. the Fourier transform of}\ \on{IC}(\cO,\cL))
\ \text{appears in} 
\\
(\pi_{\cL,1})_*\on{IC}(\on{Hess}_{\cL,1},\hat\cL_1) 
\ (\on{resp.}\ \  (\pi_{\cL,-1})_*\on{IC}(\on{Hess}_{\cL,-1},\hat\cL_{-1})
)\  
\text{as a direct summand.}
\end{gathered}
\eeq

Assume from now on that $\pi_{\cL,-1}:\on{Hess}_{\cL,-1}\to\fg_1$ is surjective. Then the sheaf\linebreak $(\pi_{\cL,-1})_*\on{IC}(\on{Hess}_{\cL,-1},\hat\cL_{-1})$ is smooth over $\fg_1^{rs}$. One sees this as follows. According to the first statement of~\eqref{LY2} the characteristic variety of $(\pi_{\cL,-1})_*\on{IC}(\on{Hess}_{\cL,-1},\hat\cL_{-1})$ coincides with that of $(\pi_{\cL,1})_*\on{IC}(\on{Hess}_{\cL,1},\hat\cL_1)$ as they are Fourier transforms of each other. But\linebreak $(\pi_{\cL,1})_*\on{IC}(\on{Hess}_{\cL,1},\hat\cL_1)$ is $K$-equivariant and supported on the nilpotent cone. A straightforward calculation then shows the smoothness of $(\pi_{\cL,-1})_*\on{IC}(\on{Hess}_{\cL,-1},\hat\cL_{-1})$ on $\fg_1^{rs}$. Thus, by the decomposition theorem, we conclude that:
\beq
\text{$(\pi_{\cL,-1})_*\on{IC}(\on{Hess}_{\cL,-1},\hat\cL_{-1})|_{\Lg_1^{rs}}$ is a direct sum of shifts of irreducible local systems}\,.
\eeq
In addition, the $\on{IC}(\on{Hess}_{\cL,-1},\hat\cL_{-1})$ and hence $(\pi_{\cL,-1})_*\on{IC}(\on{Hess}_{\cL,-1},\hat\cL_{-1})$ has a canonical structure as a Hodge module and thus the direct summands are $\on{IC}$-extensions of irreducible variations of pure Hodge structure, see, \cite{S}.

We fix a generic $s\in\fg_1^{rs}$ and then 
\beq\label{stalk}
H^*((\pi_{\cL,-1})_*\on{IC}(\on{Hess}_{\cL,-1},\hat\cL_{-1}))_s \ = \ \on{IH^*}(\on{Hess}_{\cL,-1,s},\hat\cL_{-1})\,.
\eeq
Thus we obtain an action of the  fundamental group $\pi_1^K(\fg_1^{rs},s)$ on $\on{IH^*}(\on{Hess}_{\cL,-1,s},\hat\cL_{-1})$ and by the discussion above this action breaks into a direct sum of irreducible representations which are also variations of Hodge structure. 

The component group $\pi_0(Z_K(s))\cong I_N$ acts on 
$\on{IH^*}(\on{Hess}_{\cL,-1,s},\hat\cL_{-1})$ and we write  
\[\on{IH^*}(\on{Hess}_{\cL,-1,s},\hat\cL_{-1})=\bigoplus_{\chi\in I_N^\vee} \on{IH^*}(\on{Hess}_{\cL,-1,s},\hat\cL_{-1})_\chi\]
for the corresponding eigenspace decomposition.

\begin{definition}
The stable part $\on{IH^*}(\on{Hess}_{\cL,-1,s},\hat\cL_{-1})_{\on{st}}$ of 
$\on{IH^*}(\on{Hess}_{\cL,-1,s},\hat\cL_{-1})$ is the direct summand 
$\on{IH^*}(\on{Hess}_{\cL,-1,s},\hat\cL_{-1})_{\chi_{\on{triv}}}$ where $\chi_{\on{triv}}\in I_N^\vee$ is the trivial character.
\end{definition}
Observe that $\on{IH^*}(\on{Hess}_{\cL,-1,s},\hat\cL_{-1})_{\on{st}}$
is stable under the monodromy action of $\pi_1^K(\fg_1^{rs},s)$. Moreover, the action 
factors through the braid group $B_N$ via the quotient map 
$\pi_1^K(\fg_1^{rs},s)\to B_N$.

For every irreducible representation $D_\mu$ of $\cH_{N,-1}$, let $V_\mu$ be the local system on $\fg_1^{rs}$ 
associated to $D_\mu$.
By Theorem \ref{main theorem}, there exists a unique nilpotent 
pair $(\cO_\mu,\cL_\mu)$ such that 
$\fF(\IC(\cO_\mu,\cL_\mu))\cong\IC(\fg_1,V_\mu)$.

\begin{thm}\label{geometric rel of rep}
Let $D_\mu$ be an irreducible representation of $\cH_{N,-1}$ and let 
$(\cO_\mu,\cL_\mu)$ be the associated nilpotent pair as above. 
We have 
\begin{enumerate}
\item
The map $\pi_{\cL_\mu,-1}$ is onto, 
the action of 
the braid group $B_N$ on
$\on{IH^*}({\on{Hess}}_{\cL_\mu,-1,s},\hat\cL_{\mu,-1})_{\on{st}}$ factors through 
the Hecke algebra $\cH_{N,-1}$ and 
$\on{IH^*}({\on{Hess}}_{\cL_\mu,-1,s},\hat\cL_{\mu,-1})_{\on{st}}$
is a direct sum of irreducible representations of $\cH_{N,-1}$.
\item
$D_\mu$ appears in $\on{IH^*}({\on{Hess}}_{\cL_\mu,-1,s},\hat\cL_{\mu,-1})_{\on{st}}$
with non-zero multiplicity.
\end{enumerate}
\end{thm}

\begin{proof}
Since for every irreducible subrepresentation $W$ of $\on{IH^*}({\on{Hess}}_{\cL_\mu,-1,s},\hat\cL_{\mu,-1})_{\on{st}}$
the corresponding Fourier transform $\fF(\IC(\fg_1,\cW))$ is supported on
the nilpotent cone (here $\cW$ is the local system on $\fg_1^{rs}$ associated to $W$), 
the same argument as in the proof of Theorem \ref{coh of Hess} implies part (1).
Part (2) follows from (\ref{LY2}), (\ref{LY}), and (\ref{stalk}).
\end{proof}

\section{Conjecture on more precise matching}\label{sec-conj}

In Theorem~\ref{main theorem} we show that the Fourier transform establishes a bijection between two sets of intersection cohomology sheaves. In this section we formulate a conjecture which refines the bijection in Theorem~\ref{main theorem}. 
We also relate the conjecture to our earlier conjectures in \cite{CVX2}. Our conjecture is not strong enough to produce an exact matching. The exact description of the bijection is crucial for applications, for example, computing cohomologies of Hessenberg varieties as explained in Section~\ref{sec-Hess}.

We begin with associating to each nilpotent orbit $\cO_\lambda$ (resp. $\cO_\lambda^\omega$, $\omega=\I,\II$) a subset $\Sigma_\lambda\subset\Sigma_{N}$ (resp. $\Sigma_\lambda^\omega\subset\Sigma_{N}$), if $\lambda\in\cP(N)$ has at least one odd part (resp. has only even parts). Let  $\lambda$ be a partition of $N$ and let $\lambda'$ be the transpose partition of $\lambda$. Suppose that
\beq\label{transpose partition}
\lambda'=(\lambda_1')^{2m_1}\cdots(\lambda_{l}')^{2m_l}(\lambda_{l+1}')^{2m_{l+1}-1}\cdots(\lambda_{k}')^{2m_k-1},
\eeq
where $m_i\geq 1$, $i=1,\ldots,k$.
Here and in what follows we write the parts in a partition in the order which is most convenient for us. In particular, in~\eqref{transpose partition} we place the parts with even multiplicity before the parts with odd multiplicity.

Let $\delta_i\in\{0,1\}$ for $i\in[1,l]$ and let
\beqn
\nu(\delta_1,\ldots,\delta_l)=(\lambda_1')^{m_1-\delta_1}\cdots(\lambda_{l}')^{m_l-\delta_l}(\lambda_{l+1}')^{m_{l+1}-1}\cdots(\lambda_{k}')^{m_k-1},
\eeqn
\beqn
\mu(\delta_1,\ldots,\delta_l)=(\lambda_1')^{2\delta_1}\cdots(\lambda_{l}')^{2\delta_l}(\lambda_{l+1}')\cdots(\lambda_{k}').
\eeqn
Note that $2|\nu(\delta_1,\ldots,\delta_l)|+|\mu(\delta_1,\ldots,\delta_l)|=N$. Let 
\beqn
\text{$J\subset J_0:=\{l+1,\ldots,k\}$ such that $\sum_{j\in J}{\lambda'_j}<\sum_{j\in J_0-J}\lambda'_j$.}
\eeqn
 We define
\beqn
\mu^1(\delta_1,\ldots,\delta_l;J)=(\lambda_1')^{\delta_1}\cdots(\lambda_{l}')^{\delta_l}(\lambda_{j_1}')\cdots(\lambda_{j_s}'),\ J=\{j_1,\ldots,j_s\}.
\eeqn
\beqn
\mu^2(\delta_1,\ldots,\delta_l;J)=(\lambda_1')^{\delta_1}\cdots(\lambda_{l}')^{\delta_l}(\lambda_{i_1}')\cdots(\lambda_{i_{k-l-s}}'),\ J_0-J=\{i_1,\ldots,i_{k-l-s}\}.
\eeqn

Note that $\lambda_{l+1}'=0$ if and only if all parts of $\lambda$ are even. In this case, $J_0=\emptyset=J$ and $\mu^1(\delta_1,\ldots,\delta_l;J)=\mu^2(\delta_1,\ldots,\delta_l;J)$ and we write $\mu(\delta_1,\ldots,\delta_l)=\mu^i(\delta_1,\ldots,\delta_l;J)$, $i=1,2$. 

If $\lambda$ has at least one odd part, then let  
\beqn
\begin{gathered}
\Sigma_\lambda:=\{\left(\nu(\delta_1,\ldots,\delta_l);\mu^1(\delta_1,\ldots,\delta_l;J),\mu^2(\delta_1,\ldots,\delta_l;J)\right)|\,\delta_i\in\{0,1\}, i=1,\ldots,l,\\
\qquad \ J\subset \{l+1,\ldots,k\},\text{ such that }\sum_{j\in J}{\lambda'_j}<\sum_{j\in J_0-J}\lambda'_j\}.
\end{gathered}
\eeqn
If all parts of $\lambda$ are even (in which case $\lambda_{l+1}'=0$), then let
\beqn
\Sigma_\lambda^\omega=\{\left(\nu(\delta_1,\ldots,\delta_l);\mu(\delta_1,\ldots,\delta_l),\mu(\delta_1,\ldots,\delta_l)\right)^\omega|\,\delta_i\in\{0,1\}, i=1,\ldots,l\}, \ \omega=\I,\II.
\eeqn
We have $|\Sigma_\lambda|=2^{k-1}$ (resp. $|\Sigma_\lambda^\omega|=2^l$), which equals the number of non-isomorphic irreducible $K$-equvariant local systems on $\cO_\lambda$ (resp. $\cO_\lambda^\omega$).

\begin{conjecture}\label{conjecture}
Let $\lambda$ be a partition of $N$.

(1) If $\lambda$ has at least one odd part, then the Fourier transform $\fF$ induces the following bijection
\beqn
\begin{gathered}
\fF:\{\IC(\cO_\lambda,\cE)\,|\, \cE \text{ irreducible $K$-equivariant local system on $\cO_\lambda$ (up to isomorphism)}\}\\
\xrightarrow{\sim}\{\IC\left(\Lg_1^{|\nu|},\cT\left(\nu;\mu^1,\mu^2\right)\right)|\,(\nu;\mu^1,\mu^2)\in\Sigma_\lambda\}.
\end{gathered}
\eeqn
Moreover, 
\beqn
\fF\left(\IC(\cO_\lambda,\bC)\right)=\IC\left(\Lg_1^{|\nu_0|},\cT\left(\nu_0;\mu^1_0,\mu^2_0\right)\right)
\eeqn
where $(\nu_0;\mu^1_0,\mu^2_0)\in\Sigma_\lambda$ is the unique triple such that $|\nu_0|=\max\{|\nu|,(\nu,\mu^1,\mu^2)\in \Sigma_\lambda\}$ and the parts of $\mu^1_0$ and the parts of $\mu^2_0$ have the opposite parity (in particular, all parts of $\mu_0^i$ have the same parity).

(2) If all parts of $\lambda$ are even, then the Fourier transform induces the following bijection
\beqn
\begin{gathered}
\fF:\{\IC(\cO_\lambda^\omega,\cE)\,|\,\omega=\I,\II,\ \cE \text{ irreducible $K$-equivariant local system on $\cO_\lambda^\omega$ (up to isom)}\}\\
\xrightarrow{\sim}\{\IC\left(\Lg_1^{|\nu|},\cT\left(\nu;\mu,\mu\right)^\omega\right)|\,\omega=\I,\II,\ (\nu;\mu,\mu)^\omega\in\Sigma_\lambda^\omega,\ \mu\neq\emptyset\}\\
\cup \left\{\IC\left(\Lg_1^{n,\omega},\cT\left(\nu;\emptyset,\emptyset\right)\right)|\,\omega=\I,\II,\ (\nu;\emptyset,\emptyset)\in\Sigma_\lambda^\omega\right\}.
\end{gathered}
\eeqn
Moreover, 
\beqn
\fF\left(\IC(\cO_\lambda^\omega,\bC)\right)=\IC\left(\Lg_1^{n,\omega},\cT\left(\nu_0;\emptyset,\emptyset\right)\right)
\eeqn
where $|\nu_0|=n$ and $(\nu_0;\emptyset,\emptyset)\in\Sigma_\lambda$.

\end{conjecture}

Note that $\fF(\IC(\cO_\lambda,\cE))$ has full support if and only if $\nu(\delta_1,\ldots,\delta_l)=\emptyset$. Thus we see  that the conjecture is compatible with Corollary~\ref{coro-full supp2}. We also remark that special cases of the conjecture are verified by \cite[Theorems 4.1 and 4.3]{CVX1}.

Let us relate the conjecture above to our previous conjectures in~\cite{CVX2}. In \cite{CVX2} we constructed local systems $E_{i,j}^{2n+1}$ and $\widetilde{E}_{i,j}^{2n+1}$ on $\Lg_1^{rs}$. In terms of the parametrization introduced in this paper, we have
\beqn
\text{$E_{i,j}^{2n+1}=\cT(\emptyset;(2i-j,\,j),(2n+1-2i))$}
\eeqn
\beqn
\text{$\widetilde{E}_{i,j}^{2n+1}=\cT(\emptyset;(2i-1-j,\,j),(2n+2-2i))$}.
\eeqn
Thus we see that Conjecture~\ref{conjecture} applied to $E_{i,j}^{2n+1}$ agrees with Conjectures 6.1 and 6.3 in \cite{CVX2}. Applied to  $\widetilde{E}_{i,j}^{2n+1}$, Conjecture~\ref{conjecture} implies that the supports of $\fF(\IC(\Lg_1,\widetilde{E}_{i,j}^{2n+1}))$ are as follows:
\beqn
\begin{gathered}
\cO_{3^j2^{2i-2j-1}1^{2n+3-4i+j}}\text{ if }4i-j\leq 2n+3,\\
\cO_{3^j2^{2n+2-2i-j}1^{4i-j-2n-3}}\text{ if }2i+j\leq 2n+2\text{ and }4i-j\geq 2n+3\\
\cO_{3^{2n-2i+2}2^{2i+j-2n-2}1^{2i-2j-1}}\text{ if }2i+j\geq 2n+2.
\end{gathered}
\eeqn
Note that the above orbits are all of even dimensional and  each of the even-dimensional orbits appears twice there.

\end{document}